\documentclass[amsfonts,reqno]{amsart}
\usepackage[margin=1.2in,marginpar=1in]{geometry}
\usepackage[utf8]{inputenc}
\usepackage[T1]{fontenc}
\usepackage[british]{babel}
\usepackage{
  lmodern,    
  microtype,  
  amssymb,
  mathtools,
  paralist,
  tabularx,
  booktabs,
  changepage, 
}
\mathtoolsset{mathic}

\usepackage[all,cmtip]{xy}
\usepackage[font=small,format=plain,width=\textwidth]{caption}
\usepackage[dvipsnames]{xcolor}
\numberwithin{equation}{section}

\theoremstyle{theorem}
\newtheorem*{theoremA}{Theorem~A}
\newtheorem*{theoremB}{Theorem~B}
\newtheorem{theorem}[equation]{Theorem}
\newtheorem*{theorem*}{Theorem}
\newtheorem{lemma}[equation]{Lemma}
\newtheorem{proposition}[equation]{Proposition}
\newtheorem{corollary}[equation]{Corollary}

\theoremstyle{definition}
\newtheorem{definition}[equation]{Definition}
\theoremstyle{remark}
\newtheorem{remark}[equation]{Remark}
\newcommand*{\id}{\mathrm{id}}
\renewcommand*{\deg}[1]{\left\vert#1\right\vert}
\newcommand*{\modtwo}[1]{\overline{\rule{0pt}{1.75ex} #1}}

\newcommand*{\twistedbideg}[2]{(#1,\modtwo{#2})}
\newcommand*{\twistedtrideg}[3]{(#1,#2,\modtwo{#3})}
\newcommand*{\Aut}{\operatorname{Aut}}

\newcommand*{\point}{\mathrm{pt}}

\newcommand*{\ZZ}{\mathbb{Z}}
\newcommand*{\QQ}{\mathbb{Q}}
\newcommand*{\RR}{\mathbb{R}}
\newcommand*{\ZZII}{\ZZ/2}

\newcommand*{\sheaf}[1]{{\mathcal #1}}
\newcommand*{\Oo}{\sheaf{O}}
\newcommand*{\Spec}{\mathrm{Spec}}
\newcommand*{\Proj}{\mathrm{Proj}}
\newcommand*{\Pic}{\mathrm{Pic}}
\newcommand*{\LC}{\sheaf{L}}

\renewcommand*{\AA}{\mathbb{A}}
\newcommand*{\PP}{\mathbb{P}}
\newcommand*{\PPpX}{\PP^p_x}
\newcommand*{\PPpY}{\PP^p_{y}}
\newcommand*{\RP}{\mathbb{RP}}

\newcommand*{\inner}[1]{\smash{\overset{\circ}{#1}}}
\newcommand*{\twist}[2][1]{#2(#1)}         

\renewcommand*{\H}{\operatorname{H}}        
\newcommand*{\CH}{\operatorname{CH}}        
\newcommand*{\Ch}{\operatorname{Ch}}        
\newcommand*{\CW}{\smash{\widetilde{\CH}}}  
\newcommand*{\GW}{\operatorname{GW}}        
\newcommand*{\W}{\operatorname{W}}          
\newcommand*{\FI}{\operatorname{I}}         
\newcommand*{\FIb}{\operatorname{\bar{I}}}  
\newcommand*{\I}{\mathbf{I}}                
\newcommand*{\II}{\mathbb{I}}               
\newcommand*{\Ib}{\bar{\mathbf{I}}}
\newcommand{\K}{\mathrm{K}}
\newcommand*{\KMW}{\mathbf{K}^{MW}}
\newcommand*{\KM}{\mathbf{K}^{M}}
\newcommand*{\tot}{\textnormal{tot}}

\newcommand{\red}{\mathrm{red}}             

\usepackage[draft=false]{hyperref}
\hypersetup{breaklinks=true,colorlinks=true,linkcolor=black,anchorcolor=black,citecolor=black,filecolor=black,menucolor=black,urlcolor=MidnightBlue}

\allowdisplaybreaks

\begin{document}

\title{Chow-Witt rings of split quadrics}

\author{Jens Hornbostel}
\address{Jens Hornbostel,  Bergische Universit\"at Wuppertal, Gaußstraße 20, 42119 Wuppertal, Germany}
\email{hornbostel at math.uni-wuppertal.de}
\thanks{Research for this publication was conducted in the framework of the DFG Research Training Group 2240: Algebro-Geometric Methods in Algebra, Arithmetic and Topology.}

\author{Heng Xie}
\address{Heng Xie, Bergische Universit\"at Wuppertal, Gaußstraße 20, 42119 Wuppertal, Germany}
\email{sysuxieheng at gmail.com}
\thanks{The second author was supported by EPSRC Grant EP/M001113/1 and by the DFG Priority Programme 1786: Homotopy Theory and Algebraic Geometry.}

\author{Marcus Zibrowius}
\address{Marcus Zibrowius, Heinrich-Heine-Universit\"at D\"usseldorf, Universitätsstr.~1, 40225 Düsseldorf, Germany}
\email{marcus.zibrowius at cantab.net}

\maketitle

\begin{abstract}
  We compute the Chow-Witt rings of split quadrics over a field of characteristic not two.  We even
  determine the full bigraded $\I$-cohomology and
  Milnor-Witt cohomology rings, including twists by line bundles.
  The results on $\I$-cohomology corroborate the general philosophy that
  $\I$-cohomology is an algebro-geometric version of singular cohomology of real varieties:
  our explicit calculations confirm that the $\I$-cohomology ring of a split quadric over the
  reals is isomorphic to the singular cohomology ring of the space of its real points.  
\end{abstract}

\tableofcontents

\section{Introduction}
Chow-Witt groups were introduced by Barge and Morel \cite{BM00} as a cohomology theory containing an {\em Euler class}, detecting if a vector bundle over a smooth affine  scheme over a field of characteristic \(\neq 2\) splits off a trivial line bundle in the critical range. This generalized a previous result of Murthy \cite{Mu94} for Chow groups of smooth affine schemes over algebraically closed fields.  In recent years, subsequent work of Morel \cite{MField}, Fasel \cite{faselthesis}, Fasel-Srinivas \cite{FaselSrinivas} and Asok-Fasel \cite{AF} has completed the picture if two is  invertible.  As Chow-Witt groups are a quadratic refinement of Chow groups, one might hope for generalizations of other results using Chow-Witt groups in the future.

In this paper, we determine the Chow-Witt ring of a split quadric \(Q_n\) of dimension \(n \geq 3\) over a field \(F\) of characteristic \(\neq 2\). Our computation follows the general strategies laid out by Fasel, Wendt and the first author in the computations for projective spaces, Grassmannians and classifying spaces \cite{fasel:ij,HornbostelWendt,W18}.  The main step (in both descriptions) is thus the computation of $\I$-cohomology. We compute these cohomology groups in all bidegrees, and over an arbitrary smooth base (not just a field).  That is, we compute the group
\[
  \H^i(Q_n,\I^j,\Oo(l))
\]
for all values of \(i\), \(j\) and \(l\), where \(Q_n\) is a split quadric over a smooth scheme, and where \(\Oo(l)\) denotes the \(l^{\text{th}}\) tensor power of \(\Oo_Q(1)\) (see Section~\ref{sec:notation} for precise definitions).  This additive result is summarized in Theorem~\ref{thm:I-additive} below.

For the computation of the Chow-Witt ring, we need to understand the $\I$-cohomology groups in ``geometric'' bidegrees. With untwisted coefficients, their direct sum yields the graded ring \(\H^\bullet(Q_n, \I^\bullet):= \bigoplus_{i\geq 0} \H^i(Q_n, \I^i)\), which is a graded commutative algebra over the Witt ring \(\W(F)\) of the base field.  Adding twisted coefficients, this extends to a \(\ZZ \oplus \ZZII\)-graded \(\W(F)\)-algebra
\[
  \H^\bullet(Q_n, \I^\bullet, \Oo \oplus \Oo(1)) := \bigoplus\limits_{i\geq 0} \left(\H^i(Q_n, \I^i, \Oo) \oplus \H^i(Q_n, \I^i, \Oo(1))\right).
\]
The explicit ring structure can often be written down more concisely when twisted coefficients are taken into account.  Indeed, this is already true for projective spaces.  The following result is Proposition~4.5 in a recent preprint of Matthias Wendt \cite{W18}, based on the additive computations of Fasel \cite{fasel:ij}.
\begin{theorem*}[Fasel/Wendt]
  Let \(F\) be field of characteristic \(\neq 2\), with Witt ring \(\W(F)\) and fundamental ideal \(\FI(F)\subset \W(F)\). Let \(\PP^p\) denote the \(p\)-dimensional projective space over~\(F\).
  For any \(p \geq 1\), we have a \(\ZZ \oplus \ZZII\)-graded ring isomorphism
  \begin{align*}
    \H^\bullet(\PP^p, \I^\bullet, \Oo \oplus \Oo(1))
    &= \W(F)[\xi,\alpha]/(\FI(F)\xi, \xi^{p+1},\xi\alpha,\alpha^2)
      \quad\text{ with }
      \begin{cases}
	\deg{\xi} = \twistedbideg{1}{1}\\
	\deg{\alpha} = \twistedbideg{p}{p+1}
      \end{cases}
  \end{align*}
\end{theorem*}

We generalize this result for projective spaces to the full \(\ZZ\oplus\ZZ\oplus\ZZII\)-graded setting in Theorem~\ref{alternativeproof} below.  
For split quadrics, we obtain the following corresponding result (Theorems~\ref{thm:I-ring} and \ref{thm:I-ring-2}):

\begin{theoremA}
  For any \(n \geq 3\), we have a \(\ZZ \oplus \ZZII\)-graded ring isomorphism
  \begin{align*}
    \H^\bullet(Q_n, \I^\bullet, \Oo \oplus \Oo(1))
    &\cong 
      \begin{cases}
        \W(F)[\xi,\alpha,\beta]/(\FI(F)\xi,\xi^{p+1},\xi\alpha+\xi\beta,\alpha^2 + \beta^2,\alpha\beta) & \text{if \(n = 2p\) and \(p\) is even } \\
        \W(F)[\xi,\alpha,\beta]/(\FI(F)\xi,\xi^{p+1},\xi\alpha+\xi\beta,\alpha^2,\beta^2) & \text{if \(n =2p \) and \(p\) is odd } \\
        \W(F)[\xi,\alpha,\beta]/(\FI(F)\xi,\xi^{p+1},\xi\alpha,\alpha^2,\beta^2) & \text{if \(n=2p+1\)}
      \end{cases}
  \end{align*}
  Here, the generators are of degrees
  \(\deg{\xi}    = \twistedbideg{1}{1}\),
  \(\deg{\alpha} = \twistedbideg{p}{p+1}\),
  \(\deg{\beta}  = \twistedbideg{q}{q+1}\) where \(q = p\) if \(n=2p\) is even and \(q = p+1\) if \(n=2p+1\) is odd.
\end{theoremA}

In fact, Theorems \ref{thm:I-ring} and \ref{thm:I-ring-2} again describe the full \(\ZZ\oplus\ZZ\oplus\ZZII\)-graded \(\I\)-cohomology of \(Q_n\).  Finally, for the Chow-Witt ring, we arrive at two related presentations.  The first presentation, discussed in Section~\ref{sec:CW-ring}, looks as follows. 
\begin{theoremB}
  The \(\ZZ\oplus\ZZII\)-graded \(\GW(F)\)-algebra \(\CW^\bullet(Q_{n},\Oo \oplus \Oo(1))\)  can be described as a fibre product of \(\ZZ\oplus\ZZII\)-graded rings of the following form:
  \[
    \CW^\bullet(Q_{n},\Oo \oplus \Oo(1))\stackrel{\cong}{\to} \H^\bullet(Q_{n},\I^\bullet,\Oo \oplus \Oo(1))\times_{\Ch^\bullet(Q_{n},\Oo\oplus\Oo(1))} (\ker \partial \oplus \ker \partial_{\Oo(1)}).
  \]
  Here, the second factor \(\ker \partial \oplus \ker \partial_{\Oo(1)}\) is described explicitly in Theorem~\ref{thm:CW-ring}.  The fibre product is taken over the \(\ZZ\oplus\ZZII\)-graded ring \(\Ch^{\bullet}(Q_{n},\Oo\oplus\Oo(1)) := \Ch^\bullet(Q_{n})[\tau]/(\tau^2-1)\), where \(\Ch^\bullet(Q_{n})\) is the (well-known) Chow-ring modulo two of \(Q_n\), concentrated in degrees \((*,0)\), and \(\tau\) is an artificial generator of bidegree \((0,1)\).
\end{theoremB}

The second presentation of the Chow-Witt ring is discussed in the final section.  It gives a full additive description of the \(\ZZ \oplus \ZZ \oplus \ZZII\)-graded Milnor-Witt cohomology \(\H^\tot(Q_n, \KMW)\) including all twists, refining \cite[Theorem~11.7]{fasel:ij} from projective spaces to split quadrics. The arguments presented in Section~\ref{multiplicativecomp} can also be refined to obtain a description of \(\H^\tot(Q_n, \KMW)\) as an algebra in terms of generators and relations over the coefficient ring \(\H^\tot(F, \KMW)\), but we do not make this explicit here.

\medskip

For a real variety \(X\), $\I$-cohomology is closely related to the singular cohomology of the real points \(X(\RR)\) equipped with the analytic topology.  The strongest general result in this direction is the following theorem of Jacobson \cite{jacobson}:
\begin{theorem*}[Jacobson]
  For a real variety \(X\), the group \(\H^i(X,\I^j)\) is isomorphic to the singular cohomology group \(\H^i(X(\RR),\ZZ)\) in all bidegrees with \(j > \dim(X)\). 
\end{theorem*}
Note that this general result does not apply to the interesting geometric bidegrees \(i=j\) with \(i\leq \dim(X)\).
However, the known computations for projective spaces and Grassmannians indicate that the situation for cellular varieties is even better.  In forthcoming work with Wendt \cite{HWXZ}, we will show that realization induces a ring isomorphism \(\H^\bullet(X,\I^\bullet)\cong \H^\bullet(X(\RR),\ZZ)\) for all smooth cellular real varieties \(X\).  This isomorphism also extends to twisted coefficients.  Here, we verify this result for split real quadrics, through explicit computations.

Unfortunately, we could not find the integral cohomology ring of a real split quadric in the literature.  All we found was the additive structure of cohomology with \(\ZZII\)-coefficients, which can be computed using the Gysin sequence, see \cite[\S\,24.39, problem 10]{Dieudonne}, \cite{WuLi} or \cite{CGT}.  We therefore include computations of the integral cohomology rings of real quadrics as a preliminary section in this paper.  These computations cover all real quadrics, not just the split ones, and also include twisted coefficients.  They are completely independent of the computations of $\I$-cohomology that follow, thus providing a baseline.  The results are summarized in Theorem~\ref{top:integral-ring}.  The computations for split quadrics extend from Chow groups not only to Milnor cohomology (Corollary~\ref{coro:MilnorQ}), but also to full bigraded motivic cohomology \cite[Prop. 4.3]{DI:Hopf}.  See \cite{kerz} or \cite[Section~3]{RS} for a comparison of these two bigraded theories. We expect that our techniques and results on Milnor-Witt cohomology groups of split quadrics also extend to the full bigraded ``Milnor-Witt motivic cohomology'' (also known as ``generalized motivic cohomology'') as developed by Fasel et al., see e.g.\ \cite[Theorem~4.2.4]{DF17} and \cite{calmesfasel}, but we have not pursued this.

There are many partial results on the Chow groups of quadrics by Swan, Karpenko, Merkurjev, Rost, Vishik and others. Partial results on Witt groups of split quadrics of Nenashev \cite{Ne} can be completed using the blowup setup of Balmer-Calm\`{e}s that we also use in our computations here  \cite{Ca08}. Witt groups of split quadrics over the complex numbers have been computed by the third author via another method, see \cite{Zib11}. 
We hope to return to computations of the \(\I\)-cohomology of other kinds of quadrics in future work. The trace map on Clifford algebras has been found very helpful for studying Witt groups of general quadrics, cf.\ the recent work of the second author \cite{Xie19}.

\subsection{Acknowledgements} We would like to thank the anonymous referee for carefully reading our manuscript and suggesting many improvements.

\section{Conventions and notation}\label{sec:notation}
\subsection{Topology}
We fix the following notation:
\begin{itemize}
\item \(Q_{p,q}\) is the \((p+q)\)-dimensional \textbf{real quadric} defined as a subspace of \(\RP^{p+q+1}\) by the equation
  \begin{equation}\label{eq:Q-squares}
    x_0^2 + x_1^2 + \cdots + x_p^2 = y_0^2 + y_1^2 + \cdots + y_q^2,
  \end{equation}
  where \([x_0:\ldots:x_p:y_0:\ldots:y_q]\) are homogeneous coordinates on \(\RP^{p+q+1}\). Note that  \(Q_{p,q}\cong Q_{q,p}\). We always assume \(1 \leq p\leq q\). 
\end{itemize}
Additional notation is introduced at the beginning of Section~\ref{sec:real-quadrics}.

\subsection{Algebraic geometry}\label{notalggeom}
We fix a base scheme \(S\) which is smooth over a field \(F\) of characteristic \(\neq 2\). More precisely, \(S\) is assumed to be noetherian, separated, smooth and of finite type over \(F\).  We may and will moreover assume without loss of generality that \(S\) is connected.   These are precisely the assumptions on the base scheme \(S\) used in \cite{fasel:ij}.  The following schemes and morphisms are all defined over~\(S\):

\begin{itemize}
\item \(Q_n\) denotes the \(n\)-dimensional \textbf{split quadric}, i.e.\ the closed subvariety of \(\PP^{n+1}\) defined by the equation
  \begin{equation}\label{eq:Q-split}
    \begin{cases}
      x_0y_0 + \ldots + x_p y_p = 0  & \text{ when } n = 2p	\\
      x_1y_1 + \ldots + x_p y_p + z^2 = 0 & \text{ when }  n = 2p+1
    \end{cases}
  \end{equation}
  When \(n = 2p\), \(Q_n\) is isomorphic to the subvariety \(Q_{p,p}\) defined by equation \eqref{eq:Q-squares}. When \(n= 2p+1\), it is isomorphic to \(Q_{p,p+1}\).
\item \(\PPpX\), \(\PPpY\), \(\PP_{x'}^p\) and \(\PP_{y'}^p\) denote the following subvarieties of \(Q_n\):
  
  When \(n = 2p\), we define \(\PPpX,\PPpY \subset \PP^{2p+1}\) by the equations \(y_0 = y_1 = \ldots = y_p = 0 \) and \(x_0 = x_1 = \ldots = x_p = 0 \) respectively.  We define \(\PP_{x'}^p, \PP_{y'}^p \subset\PP^{2p+1}\) by the equations \(x_0 = y_1 = \ldots = y_p = 0 \) and \(y_0 = x_1 = \ldots = x_p = 0 \). 

  When \(n = 2p+1\), we define \(\PPpX, \PPpY \subset \PP^{2p+2}\) by the equations \(y_0 = y_1 = \ldots = y_p = z = 0 \) and \(x_0 = x_1 = \ldots = x_p = z = 0\), respectively.
\item 
  \(q \colon Q_n \rightarrow S \) and \(p \colon \PP^p \to S\) denote the projection maps
\item 
  \(i_y\colon \PPpY \hookrightarrow Q_n-\PPpX\), \(\iota_y\colon \PPpY \hookrightarrow Q_n\), \(j\colon  Q_n - \PPpX \hookrightarrow Q_n\) and \(i\colon Q_n \hookrightarrow \PP^{n+1}\) are the obvious embeddings; likewise for \(i_x\), \(\iota_x\), \(i_{x'}\), \(\iota_{x'}\), \(i_{y'}\) and \(\iota_{y'}\)
\item
  \(\rho\colon Q_n -\PPpX \to \PPpY\) is the morphism given by 
  \[
    \begin{cases}
      \rho\colon [x_0:\ldots:x_p:y_0:\ldots:y_p]~ \mapsto~ [0:\ldots:0:y_0:\ldots:y_p]& \text{ if } n = 2p\\
      \rho\colon [x_0:\ldots:x_p:y_0:\ldots:y_p:z]\mapsto [0:\ldots:0:y_0:\ldots:y_p:0] &\text{ if } n = 2p+1
    \end{cases}
  \]
  (Some of this notation is summarized in Figure~\ref{fig:geometry}.)
\item
  \(\Oo_Q(1)\) is defined as the restriction along \(i\colon Q_n\hookrightarrow \PP^{n+1}\) of the canonical line bundle \(\Oo(1)\) on \(\PP^{n+1}\). When there is no danger of confusion, we simply denote this restriction again by \(\Oo(1)\). 
\end{itemize}

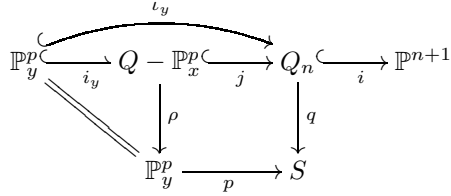
\begin{figure}[h]
  \begin{equation*}
    \begin{aligned}
      \xymatrix{
        \PPpY  \ar@^{^{(}->}[r]_{i_y} \ar@{^{(}->}@/^1pc/@<3pt>[rr]^{\iota_y} \ar@{=}[dr] & Q-\PPpX \ar[d]^{\rho} \ar@{^{(}->}[r]_-{j} & Q_n \ar[d]^q \ar@{^{(}->}[r]_{i} & \PP^{n+1} \\
        & \PPpY \ar[r]_{p} & S
      }
    \end{aligned}
  \end{equation*}
  \caption{Various morphisms used in the algebro-geometric calculations}\label{fig:geometry}
\end{figure}

\section{Singular cohomology of real quadrics}\label{sec:real-quadrics}
In this section, we study the real quadrics \(Q_{p,q}\subset \RP^{p+q+1}\) (with the analytic topology).  Note first that \(Q_{p,q}\) is a two-fold covering space of \(\RP^p\times \RP^q\), via the obvious map that takes \([x:y]\) to \(([x],[y])\).  Write \(\pi_1\) and \(\pi_2\) for the two components of this map.  A two-fold cover of the quadric itself is given by \(S^p\times S^q\):  the real quadric \(Q_{p,q}\) is the quotient of \(S^p\times S^q\) modulo the involution \(\tau := \tau_p\times\tau_q\), where \(\tau_n\) denotes the involution of \(S^n\) sending \(x\) to \(-x\).  The composition of the two covering maps is the canonical four-fold cover of \(\RP^p\times \RP^q\).  The situation is summarized by the central vertical column of Figure~\ref{fig:topology}. Also displayed there are the ``diagonal'' embedding \(\Delta^S\colon S^p \hookrightarrow S^p \times S^q\) that sends \(x\) to \((x,(x,0))\) and the two induced maps \(\Delta\colon \RP^p\to Q_{p,q}\) and \(\RP^p\to \RP^p\times \RP^q\), respectively.  Note that \(\Delta\) splits the projection \(\pi_1\colon Q_{p,q}\to \RP^p\).

\begin{figure}[h]
  \begin{equation*}
    \begin{aligned}
      \xymatrix{
        S^p \ar[dd] \ar[r]_-{\Delta^S} \ar@{=}@/^1pc/[rr]^{}
        & S^p\times S^q \ar[r]_{\pi^S_1} \ar[d]|{\strut\pi} 
        & S^p \ar[dd]\\
        & Q_{p,q} \ar@/^/[dr]^{\pi_1} \ar[d]|{\strut\pi_1\times\pi_2} \\
        \RP^p \ar[r]  \ar@{=}@/_1pc/[rr]_{} \ar@/^/[ur]^{\Delta} & \RP^p\times \RP^q \ar[r] & \RP^p
      }
    \end{aligned}
  \end{equation*}
  \caption{The various continuous maps relating spheres, projective spaces and real quadrics}\label{fig:topology}
\end{figure}
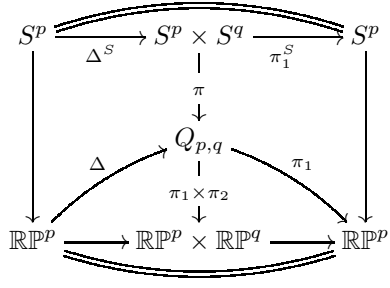

\subsection{Twisted coefficients}\label{sec:top-twists}
A free rank one local coefficient system over a path-connected space \(X\) can be defined as an action of the fundamental group on \(\ZZ\), i.e.\ as a group homomorphism \(\pi_1(X,x)\to \Aut(\ZZ)=\{\pm 1\}\), for any choice of base point \(x\) \cite[\S\,5.1]{davis-kirk}\cite[Ex.~I.F]{Spanier}. As is customary, we denote the trivial rank one local coefficient system corresponding to the constant group homomorphism \(\pi_1(X,x)\to \Aut(\ZZ)\)  simply by \(\ZZ\).  For semi-locally simply connected \(X\), such coefficient systems correspond to fibre bundles with typical fibre \(\ZZ\) and structure group \(\{\pm 1\}\) \cite[Lemma~4.7]{davis-kirk}.  Equivalently, we may identify them with principal \(\{\pm 1\}\)-bundles, i.e.\ with two-fold coverings of \(X\). 

We write \(\twist{\ZZ}\) for the rank one local coefficient system \(\rho\colon \pi_1(\RP^p)\to \{\pm 1\}\) on \(\RP^p\) corresponding to the two-fold cover \(S^p\to \RP^p\).  For \(p \geq 2\), the two-fold cover is the universal cover, and \(\rho\) is an isomorphism, but the system \(\twist{\ZZ}\) also exists for \(p = 1\).  The situation for the real quadrics is similar.  We have a canonical two-fold cover \(S^p\times S^q\to Q_{p,q}\), and we again write \(\twist{\ZZ}\) for the corresponding local coefficient system \(\pi_1(Q_{p,q}) \to \{\pm 1\}\).  Again, \(S^p\times S^q\) is the universal cover for \(q\geq p \geq 2\).  Note that \(\twist{\ZZ}\) pulls back to \(\twist{\ZZ}\) under the maps \(\pi_1\) and \(\Delta\) in Figure~\ref{fig:topology}, so the notation is consistent.

More generally, given a coefficient ring \(R\) and an integer \(s\), we define \(\twist[s]{R} := R\otimes_{\ZZ} \twist{\ZZ}^{\otimes s}\) with the action of the fundamental group induced by the trivial action on \(R\) and the above action on \(\ZZ(1)\).  Of course, this really only depends on \(R\) and the value of \(s\) mod \(2\).  Also note that \(\ZZII(s) = \ZZII\) for all~\(s\). The direct sum \(R\oplus \twist{R}\) is a \(\ZZII\)-graded ring, and hence cohomology with local coefficients in \(R\oplus \twist{R}\) is a \(\ZZ\oplus\ZZII\)-graded ring, which decomposes additively as 
\[
  \H^\bullet(Q_{p,q},R\oplus \twist{R}) = \H^\bullet(Q_{p,q},R)\oplus \H^\bullet(Q_{p,q},\twist{R}).
\]
Elements \(\alpha\in \H^i(Q_{p,q},R)\) have degree \(\deg{\alpha} = \twistedbideg{i}{0}\); elements 
\(\beta \in \H^i(Q_{p,q},\twist{R})\) have degree \(\deg{\beta} = \twistedbideg{i}{1}\).  
We sometimes also view \(R\oplus R\) as a \(\ZZII\)-graded coefficient ring for \(S^p\times S^q\) and 
\[
  \H^\bullet(S^p\times S^q,R\oplus R) = \H^\bullet(S^p\times S^q,R)\oplus \H^\bullet(S^p\times S^q,R) 
\]
as a \(\ZZ\oplus\ZZII\)-graded ring, in the evident way.

\subsection{Rational cohomology ring}\label{sec:rational-cohomology}
The description of \(Q_{p,q}\) as a quotient of \(S^p\times S^q\) implies that \(\H^\bullet(Q_{p,q},\QQ)\) is isomorphic to the subring of \(\H^\bullet(S^p\times S^q,\QQ)\) fixed by \(\tau^*\) \cite[Proposition~3G.1]{Hatcher}.  Similar arguments show that, more generally, \(\H^\bullet(Q_{p,q},\QQ\oplus\twist{\QQ})\) is isomorphic to the subring of the \(\ZZ\oplus\{\pm 1\}\)-graded ring \(\H^\bullet(S^p\times S^q, \QQ\oplus \QQ)\) 
fixed by the action of \(\tau^*\) on homogeneous elements in the \(+1\)-graded part and the action of \(-\tau^*\) on the \((-1)\)-graded part.  As \(\tau_n^*\colon \H^n(S^n) \to \H^n(S^n)\) is given by multiplication with \((-1)^{n+1}\), we find:
\begin{equation}
  \H^\bullet(Q_{p,q},\QQ\oplus \twist{\QQ}) \cong \QQ[\alpha,\beta]/(\alpha^2,\beta^2) 
  \quad\text{ with }
  \begin{cases}
    \deg{\alpha} = \twistedbideg{p}{p+1}\\
    \deg{\beta}  = \twistedbideg{q}{q+1}
  \end{cases}
\end{equation}
In particular, \(Q_{p,q}\) is orientable if and only if \(p+q\) is even (use \cite[Theorem~3.26]{Hatcher}).  

\subsection{A (minimal) cell structure}
In order to compute the twisted (co)homology of the real quadric \(Q_{p,q}\) with integral coefficients, we need to study its two-fold cover \(S^p\times S^q\) in more detail. 

\begin{proposition}\label{prop:SxS-cell-structure}
  Assume \(p\leq q\).  There exists a cell structure on \(S^p\times S^q\) consisting of \(4p+4\) cells for which the action of \(\tau:=\tau_p\times \tau_q\) is cellular: we have characteristic maps \(j_{0,0}^\pm, \dots, j_{p,0}^\pm\) for cells in dimensions \(0,\dots,p\) and characteristic maps \(j_{0,q}^\pm,\dots,j_{p,q}^\pm\) for cells in dimensions \(q,\dots,\mbox{p+q}\) such that \(\tau\circ j_{k,i}^+ = j_{k,i}^-\).  
  The skeleta of \(X=S^p\times S^q\) with respect to this cell structure are given by:
  \begin{align*}
    X^k &= \Delta^S(S^k)  && \text{ for } 0 \leq k < p\\
    X^k &= \Delta^S(S^p)  && \text{ for } p \leq k < q\\
    X^{k+q}&= \Delta^S(S^p) \cup (S^k\times S^q) &&\text{ for } q \leq k+q \leq p+q
  \end{align*}
  Here,  \(\Delta^S\colon S^p\to S^p\times S^q\) denotes the map \(x\mapsto (x,(x,0))\), and \(S^k\) is viewed as an ``equator'' of \(S^p\) for \(k\leq p\), embedded via the first \(k+1\) coordinates.
\end{proposition}
\begin{proof}
  We will give an explicit description of this cell structure on \(S^p\times S^q\).  First, let us recall and make explicit one common cell structure on a single sphere \(S^n\). For \(k\in\{0,\dots,n\}\), let \(j_k^+\) denote the following embedding of \(D^k\) into \(S^n\):
  \begin{equation*}
    \begin{aligned}
      j_k^+\colon D^k &\xrightarrow{\quad\quad} S^n\\
      (x_0,\dots, x_{k-1}) &\mapsto (x_0,\dots, x_{k-1}, \sqrt{1-\|x\|^2},0, \dots, 0)
    \end{aligned}
  \end{equation*}
  Another embedding of \(D^k\) into \(S^n\) is given by \(j_k^- := \tau_n\circ j_k^+\).
  Together, these maps \(j_k^+\) and \(j_k^-\) for \(k\in\{0,\dots,n\}\) define a cell structure on \(S^n\) with two cells in each dimension.  

  Now take \(n=p\). We give \(\Delta^S(S^p) \subset S^p\times S^q\) the cell structure induced by the cell structure on \(S^p\) just described.  That is, we consider the characteristic maps
  \begin{equation}\label{eq:defn:jk0}
    j_{k,0}^{\pm} := \Delta^S\circ j_k^{\pm}\colon D^k \to S^p\times S^q.
  \end{equation}
  Note that \(\tau \circ j_{k,0}^+ = j_{k,0}^- \) as desired.
  
  The second family of \(2p+2\) cells lives in dimensions \(q\), \dots, \(p+q\).  In order to define these cells, we need to fix some further notation. Consider the ``northern hemisphere'' \(j^+_q(D^q)\subset S^q\), i.e.\ the set of all points \(x=(x_0,\dots,x_q)\in S^q\) with \(x_q\geq 0\). Let \(e_q := (0,\dots, 0,1)\) be the ``north pole''.  
  Choose and fix a continuous map 
  \begin{align*}
    R\colon j^+_q(D^q) &\to \mathrm{SO}(q+1)\\
    x &\mapsto R_x
  \end{align*}
  with the property that \(R_x\cdot e_q = x\).  Such a map exists by Lemma~\ref{lem:rotation-families} below.
  
  Also, we fix for each dimension \(n\) a map that ``wraps the \(n\)-disk around the \(n\)-sphere'', i.e.\ a surjection \(w_n\colon D^n \to S^n\) that sends \(0\in D^n\) to the north pole \(e_n\in S^n\) and every point on the boundary \(\partial D^n\) to the south pole \(-e_n\in S^n\), and that induces a homeomorphism \(\bar w_n\colon D^n/\partial D^n \cong S^n\).  We choose \(w_n\) as follows:
  \begin{equation}\label{eq:defn:wn}
    \begin{aligned}
      w_n\colon  D^n &\to S^n\\
      x & \mapsto (s(x)\cdot x, 1 - 2\|x\|),
    \end{aligned}
  \end{equation}
  where \(s(x)\) is a non-negative scalar determined by the requirement that \(w_n(x)\in S^n\).
  Finally, fix homeomorphisms \(\phi_{m,n}\) between products of disks \(D^m\times D^n\) and the disks \(D^{m+n}\) of the appropriate dimensions:
  \begin{equation}\label{eq:defn:phimn}
    \begin{aligned}
      \phi_{m,n}\colon D^m\times D^n &\xrightarrow{\cong} D^{m+n}\\
      (x,y) &\mapsto \tfrac{\max(\|x\|,\|y\|)}{\|(x,y)\|} (x,y)
    \end{aligned}
  \end{equation}
  In view of the homeomorphisms \(\phi_{k,q}\), we may define the characteristic maps for the \((k+q)\)-cells of \(S^p\times S^q\) on the products \(D^k\times D^q\).
  For every \(k\in\{0,\dots,p\}\), we define one such characteristic map as follows:
  \begin{equation}\label{eq:defn:jkq}
    \begin{aligned}
      j_{k,q}^+\colon D^k \times D^q &\to S^p \times S^q \\
      (x,y)&\mapsto (j_k^+(x), -R_{(j_k^+(x),0)}(w_qy))
    \end{aligned}
  \end{equation}
  Here, the entry \(j_k^+(x)\) in the first coordinate denotes the embedding \(j_k^+\colon D^k\to S^p\), and the entry \((j_k^+(x),0)\) in the second coordinate denotes the composition of \(j_k^+\) with the embedding \(x\mapsto (x,0)\) of \(S^p\) into \(S^q\) as an ``equator''.
  We define a second characteristic map for a \((k+q)\)-cell as \(j_{k,q}^-:= \tau\circ j_{k,q}^+\).

  When restricted to the inner points of \(D^k\times D^q\), the characteristic maps \(j_{k,q}^\pm\) induce homeomorphisms:
  \[
    \inner{D}^k\times\inner{D}^q \xrightarrow{\;\cong\;} (j_k^\pm(\inner{D}^k)\times S^q)\setminus \Delta^S(S^p)
  \]
  The images of these homeomorphisms clearly constitute a cover of \((S^p\times S^q)\setminus\Delta^S(S^p)\) by disjoint sets.  Thus, altogether the maps \(j_{k,i}^\pm\) for \(i \in \{0,q\}\) and \(k\in\{0,\dots,p\}\) defined in \eqref{eq:defn:jk0} and \eqref{eq:defn:jkq} constitute a cell structure on \(S^p\times S^q\).
\end{proof}

\begin{lemma}\label{lem:rotation-families}
  There exists a continuous map 
  \(
  R\colon j^+_q(D^q) \to \mathrm{SO}(q+1),
  x\mapsto R_x
  \)
  such that \(R_x\cdot e_q = x\) for all~\(x\).
\end{lemma}
\begin{proof}
  Consider the evaluation map \(\mathrm{SO}(q+1)\to S^q\) that takes a matrix \(R\) to \(R\cdot e_q\).  This evaluation map constitutes a principal \(\mathrm{SO}(q)\)-bundle.  Over the contractible space \(j^+_q(D^q)\cong D^q\), this bundle is necessarily trivial, hence it admits a section.
\end{proof}

\subsection{Cellular (co)homology}
Let \(\ZZ[\tau]\) denote the group ring associated with the group with two elements, i.e.\ \(\ZZ[\tau] := \ZZ \oplus \ZZ\tau\) with \(\tau^2 = 1\).  Let \(\widetilde C^{(n)}(\modtwo{0})\) denote the chain complex of free \(\ZZ[\tau]\)-modules 
\[
  \widetilde C^{(n)}(\modtwo 0)\colon 
  \ZZ[\tau]\xleftarrow{1-\tau} 
  \ZZ[\tau]\xleftarrow{1+\tau}
  \ZZ[\tau]\xleftarrow{1-\tau}
  \ZZ[\tau]\xleftarrow{1+\tau}
  \cdots \leftarrow
  \ZZ[\tau]
\]
concentrated in degrees \(0\) to \(n\).  This is the cellular chain complex of \(S^n\) with respect to the cell structure arising from the two-fold covering map \(S^n\to \RP^n\) and the standard cell structure on \(\RP^n\).  Multiplication by \(\tau\) on the complex is the chain map induced by the involution \(\tau_n\) of \(S^n\). The usual homology of \(\RP^n\) with coefficients in \(\ZZ\) is the homology of \(\widetilde C^{(n)}(\modtwo 0)\otimes_{\ZZ[\tau]}\ZZ\), where \(\ZZ\) is viewed as trivial \(\ZZ[\tau]\)-module; if instead we tensor with the \(\ZZ[\tau]\)-module \(\twist{\ZZ}\) on which \(\tau\) acts as \(-1\), we obtain the homology of \(\RP^n\) with twisted coefficients.

Let \(\widetilde C^{(n)}(\modtwo 1)\) denote the complex with the same groups in each degree, but with the roles of multiplication by \(1-\tau\) and multiplication by \(1+\tau\) reversed:
\[
  \widetilde C^{(n)}(\modtwo 1)\colon 
  \ZZ[\tau]\xleftarrow{1+\tau} 
  \ZZ[\tau]\xleftarrow{1-\tau}
  \ZZ[\tau]\xleftarrow{1+\tau}
  \ZZ[\tau]\xleftarrow{1-\tau}
  \cdots \leftarrow
  \ZZ[\tau]
\]
\begin{proposition}\label{cellular-complex}
  For \(p\leq q\), the cellular chain complex associated with the cell structure on \(S^p\times S^q\) described in Proposition~\ref{prop:SxS-cell-structure} has the form
  \(
  C(S^p\times S^q) \cong \widetilde C^{(p)}(\modtwo 0) \oplus \widetilde C^{(p)}(\modtwo{q+1})[q]
  \),
  where \(\modtwo{q+1}\) indicates the value of \(q+1\) modulo two, and \([q]\) indicates that the second summand is shifted \(q\) degrees to the right.  Multiplication by \(\tau\) on the complex is the chain map induced by the involution \(\tau\) on \(S^p\times S^q\).
\end{proposition}

\begin{proof}[Proof sketch]
  We clearly have a decomposition 
  \[
    C(S^p\times S^q) = C(\Delta^S(S^p)) \oplus C'
  \]
  for some complex \(C'\) concentrated in degrees \(q,\dots,p+q\), except that, a priori, there might be some differentials from \(C'\) to \(C(\Delta(S^p))\). However, there are no such differentials.  This can easily be seen by considering the map on chain complexes induced by the projection \(\pi_1^S\colon S^p\times S^q\to S^p\):  this map is zero on \(C'\) and maps \(C(\Delta^S(S^p))\) isomorphically to \(C(S^p)\). Thus, the above decomposition is an honest decomposition of chain complexes.

  We may identify \(C(\Delta^S(S^p))\) with \(C(S^p)\) with respect to the usual cell structure on \(S^p\) consisting of two cells in each dimension.  This is precisely the cell structure used in the usual computation of the homology of \(\RP^p\) (cf.\  \cite[5.2.1]{davis-kirk}).  The complex \(C'\) can be computed in an analogous fashion.  The non-zero boundary maps \(d\) of \(C'\) are \(\ZZ[\tau]\)-linear maps between free \(\ZZ[\tau]\)-modules of rank one, each with a generator \(\{j^+_{k,q}\}\) represented by one of the characteristic maps \(j_{k,q}^+\). 
  Write \(d\colon C'_{k+q} \to C'_{k-1+q}\) as 
  \[
    d\{j^+_{k,q}\} = d_+\{j^+_{k-1,q}\} + d_-\tau\{j^+_{k-1,q}\}.
  \]
  Both coefficients \(d_+\) and \(d_-\) can easily be seen to be \(\pm 1\).  The crucial value we need to know for the homology/cohomology calculations is the \emph{relative} sign \(d_+/d_-\in\{\pm 1\}\).

  Recall how the coefficients of the boundary map \(d\) in the cellular chain complex of a CW complex \(X\) are defined.  Let \(X^i\) denote the \(i\)-skeleton. Given a characteristic map \(j\colon D^i\to X^i\), write \(\pi(j)\colon X^i\twoheadrightarrow S^i\) for the map that sends the complement of the open cell defined by \(j\) to the ``south pole'' \(-e_i\) and restricts to the ``wrapping map'' defined in \eqref{eq:defn:wn} on the open cell itself: \(\pi(j)\circ j=w_i\). Given an \((i-1)\)-cell with characteristic map \(j'\), the coefficient of \(\{j'\}\) in \(d\{j\}\) is the degree of the composition \(\pi(j')\circ j_{|\partial D^i}\).

  So consider \(X=S^p\times S^q\) with our chosen cell structure.  The coefficients \(d_\pm\) are the degrees of the following two compositions:
  \begin{align}
    \label{eq:degree-map+}
    f^{++}\colon \partial(D^k\times D^q) &\xrightarrow{{j^+_{k,q}}_{|\partial(D^k\times D^q)}}  X^{k-1+q} \xrightarrow{\pi(j^+_{k-1,q})} S^{k-1+q} \\
    \label{eq:degree-map-}
    f^{+-}\colon \partial(D^k\times D^q) &\xrightarrow{{j^+_{k,q}}_{|\partial(D^k\times D^q)}}  X^{k-1+q}  \xrightarrow{\pi(j^-_{k-1,q})}  S^{k-1+q}
  \end{align}
  In order to compare the degrees of these maps, we first describe them more explicitly.  Recall
  from \eqref{eq:defn:phimn} our notation \(\phi_{m,n}\) for the homeomorphism \(D^m\times D^n\cong D^{m+n}\), and let \(\bar\phi_{m,n}\) denote the induced homeomorphism
  \[
    \bar\phi_{m,n}\colon \frac{D^m\times D^n}{\partial(D^m\times D^n)} \xrightarrow{\quad\cong\quad} \frac{D^{m+n}}{\partial D^{m+n}}.
  \]
  Also recall that \(\bar w_n\) denotes the homeomorphism \(D^n/\partial D^n\to S^n\) induced by the wrapping map \(w_n\).  Let 
  \(
  w_{m,n}\colon D^m\times D^n \to S^{n+m}
  \)
  denote the composition \(w_{m,n} := w_{m+n}\circ \phi_{m,n}\), and let \(\bar w_{m,n}\) denote the following composition in which the first map is the canonical surjection from the product to the smash product:
  \[
    \bar w_{m,n} \colon
    \frac{D^m}{\partial D^m} \times \frac{D^n}{\partial D^n}
    \xrightarrow{\quad\quad}
    \frac{D^m}{\partial D^m}\wedge \frac{D^n}{\partial D^n}
    =
    \frac{D^m\times D^n}{\partial(D^m\times D^n)}  \xrightarrow[\bar \phi_{m,n}]{\quad\cong\quad}
    \frac{D^{m+n}}{\partial D^{m+n}}
    \xrightarrow[\bar w_{m+n}]{\quad\cong\quad}
    S^{m+n}
  \]
  The maps \(\pi(j^\pm_{k,q})\colon X^{k+q}\to S^{k+q}\) can be described as follows:
  \begin{align*}
    \pi(j^+_{k,q})(x,y)
    & =
      \begin{cases}
        \bar w_{k,q}\left(x_0,\dots,x_{k-1},\bar w_q^{-1}\left(-R_{(x,0)}^{-1}(y)\right)\right)
        \phantom{-()_-}
        & \text{ if  \(x_k\geq 0\) and \((x,0)\neq y\)}\\
        -e_{k+q}
        & \text{ otherwise}
      \end{cases}
    \\
    \pi(j^-_{k,q})(x,y)
    & =
      \begin{cases}
        \bar w_{k,q}\left(-(x_0,\dots,x_{k-1}),\bar w_q^{-1}\left(-R_{-(x,0)}^{-1}(y)\right)\right)
        & \text{ if  \(x_k\leq 0\) and \((x,0)\neq y\)}\\
        -e_{k+q}
        & \text{ otherwise}
      \end{cases}
  \end{align*}
  In either line, the coordinates are \(x = (x_0,\dots,x_k,0,\dots,0)\in S^p\) and \(y=(y_0,\dots,y_q)\in S^q\).  Note that only the first \(k\) coordinates of \(x\), i.e.\ the coordinates \((x_0,\dots,x_{k-1})\in D^k\), appear as arguments of \(\bar w_{k,q}\).  For the maps \(f^{+\pm}\colon \partial(D^k\times D^q) \to S^{k+q-1}\) in \eqref{eq:degree-map+} and \eqref{eq:degree-map-}, we obtain the following description:
  \begin{align}
    \label{eq:formula:f++}
    f^{++}(x,y)
    &=
      \begin{cases}
        \bar w_{k-1,q}\left(x_0,\dots,x_{k-2},y\right)
        \phantom{-()\bar w_q^{-1}\left(-R^{-1}_{-(x,0)}R_{(x,0)}\bar w_q()\right)}
        & \text{ if \(x_{k-1}\geq 0\) and \(\|x\|=1\)}\\
        -e_{k-1+q}
        & \text{ otherwise}
      \end{cases}
    \\
    \label{eq:formula:f+-}
    f^{+-}(x,y)
    &=
      \begin{cases}
        \bar w_{k-1,q}\left(-(x_0,\dots,x_{k-2}),\bar w_q^{-1}\left(-R^{-1}_{-(x,0)}R_{(x,0)} w_q(y)\right)\right)
        & \text{ if \(x_{k-1}\leq 0\) and \(\|x\|=1\)}\\
        -e_{k-1+q}
        & \text{ otherwise}
      \end{cases}
  \end{align}
  Here, the coordinates in either line are \(x = (x_0,\dots,x_{k-1})\in D^k\) and \(y = (y_0,\dots,y_{q-1})\in D^q\).  
  
  The map \(f^{+-}\) we have just computed is homotopic to the following simpler map:
  \begin{align}\label{eq:defn:f+-tilde}
    \widetilde f^{+-}(x,y)
    &=
      \begin{cases}
        \bar w_{k-1,q}\left(-(x_0,\dots,x_{k-2}),(-1)^{q+1} y_0,y_1,\dots,y_{q-1}\right)
        & \text{ if \(x_{k-1}\leq 0\) and \(\|x\|=1\)}\\
        -e_{k-1+q}
        & \text{ otherwise}
      \end{cases}
  \end{align}
  To construct a homotopy \(f^{+-} \leadsto \widetilde f^{+-}\), first consider the following map:
  \begin{align*}
    S\colon j^-_{q-1}(D^{q-1}) \to \mathrm{O}(q+1)\\
    x \mapsto S_x := -R^{-1}_{-x}R_x
  \end{align*}
  Here, \(x\) has coordinates \((x_0,\dots,x_{q-1},0)\in D^q\) with \(x_{q-1} \leq 0\).  Note that both \(x\) and \(-x\) lie in \(j^+_q(D^q)\), so that both \(R_x\) and \(R_{-x}\) are defined.  It follows from our description of \(R\) in Lemma~\ref{lem:rotation-families} that \(S_x e_q = e_q\) for each \(x\).  So the map \(S\) factors through the inclusion of \(\mathrm{O}(q)\) into \({O}(q+1)\) as ``upper left block''.  Let \(x\mapsto S'_x\in \mathrm{O}(q)\) denote the restricted map: 
  \[\xymatrix{
      j^-_{q-1}(D^{q-1})\ar[r]_-{S'} \ar@/^1em/[rr]^{S}&\mathrm{O}(q) \ar@{^{(}->}[r] & \mathrm{O}(q+1)
    }\]
  As \(j^-_{q-1}(D^{q-1})\) is contractible, the map \(S'\) is nullhomotopic.  As \(R_x\) always has determinant one, the determinant of \(S'_x\) can be computed as \(\det(S'_x) = \det(S_x) = \det(-\id_{\RR^{q+1}}) = (-1)^{q+1}\).  We may therefore choose a homotopy from \(S'\) to the constant map \(j^-_{q-1}(D^{q-1})\to \mathrm{O}(q)\) with value the  \((q\times q)\)-diagonal matrix \(T_q^{q+1}\), where
  \[
    T_q^i :=
    \begin{pmatrix}
      (-1)^i & & & \\
      & 1 & & \\
      &  & \ddots & \\
      & & & 1 
    \end{pmatrix} .
  \]
  Consequently, \(S\) is homotopic to the map \(j^-_{q-1}(D^{q-1})\to \mathrm{O}(q+1)\) with constant value \((T_{q+1})^{q+1}\). Denote this homotopy by \(S(t)\):
  \begin{align*}
    [0,1]\times j^-_{q-1}(D^{q-1}) &\to \mathrm{O}(q+1)\\
    (x,t) \quad& \mapsto S_x(t)
  \end{align*}
  Then \(S_x(0) = S_x\), \(S_x(1) = (T_{q+1})^{q+1}\), and \(S_x(t)e_q = e_q\) for all \(x\) and all \(t\).  Now consider the homotopy \(f\colon [0,1]\times \partial(D^k\times D^q) \to  S^{k-1+q}\) defined as follows:
  \begin{align*}
    f(t)(x,y) := 
    \begin{cases}
      \bar w_{k-1,q}(-(x_0,\dots,x_{k-2}),\bar w_q^{-1}(S_{(x,0)}(t)\cdot w_q(y)))
      & \text{ if \(x_{k-1}\leq 0\) and \(\|x\|=1\)}\\
      -e_{k-1+q}
      & \text{ otherwise}
    \end{cases}
  \end{align*}
  This is the homotopy from \(f^{+-}\) to \(\widetilde f^{+-}\) that we need.  Indeed, the equality \(f(0)=f^{+-}\) is obvious, and the equality \(f(1) = \widetilde f^{+-}\) easily follows once we observe that \(T_{q+1}w_q(y) = w_q(T_q y)\) for our choice of wrapping map \(w_q\) (see \eqref{eq:defn:wn}).

  It follows in particular that \(f^{+-}\) and \(\widetilde f^{+-}\) have the same degree.  On the other hand, we see from the formulas given in \eqref{eq:formula:f++} and \eqref{eq:defn:f+-tilde} that \(\widetilde f^{+-} = f^{++}\sigma\) for the involution \(\sigma\colon \partial(D^k\times D^q)\to\partial(D^k\times D^q)\) given by 
  \[
    \sigma(x_0,\dots,x_{k-1},y_0,\dots,y_{q-1}) := (-x_0,-x_1,\dots,-x_{k-1},(-1)^{q+1}y_0,y_1,\dots, y_{q-1}).
  \]
  So altogether we find that
  \newcommand{\degree}{\mathrm{deg}}
  \(
  d_+/d_- = \degree(f^{++})/\degree(\widetilde f^{+-})
  = \degree(\sigma) = (-1)^{k+q+1}
  \).
  This shows that the cellular chain complex of \(S^p\times S^q\) has the form claimed.
\end{proof}

Now let \(\ZZ(0) = \ZZ\) and \(\twist{\ZZ}\) denote the trivial \(\ZZ[\tau]\)-module and the \(\ZZ[\tau]\)-module on which \(\tau\) acts as multiplication by \(-1\), respectively. 
Let \(C^{(n)}(\modtwo t)\) denote the following chain complexes of abelian groups:
\begin{align*}
  C^{(n)}(\modtwo 0)\colon &
                             \ZZ\xleftarrow{0}
                             \ZZ\xleftarrow{2}
                             \ZZ\xleftarrow{0}
                             \ZZ\xleftarrow{2}
                             \cdots \leftarrow
                             \ZZ
  \\
  C^{(n)}(\modtwo 1)\colon &
                             \ZZ\xleftarrow{2}
                             \ZZ\xleftarrow{0}
                             \ZZ\xleftarrow{2}
                             \ZZ\xleftarrow{0}
                             \cdots \leftarrow
                             \ZZ
\end{align*}
These are related to the chain complexes of \(\ZZ[\tau]\)-modules considered above by canonical isomorphisms \(\widetilde C^{(n)}(\modtwo t) \otimes_{\ZZ[\tau]} \twist[s]{\ZZ} \cong C^{(n)} (\modtwo{s+t})\). 
\begin{corollary}
  Assume \(p\leq q\).  There exists a cell structure on \(Q_{p,q}\) with \(2p+2\) cells: one \(i\)-cell for each \(i\in\{0,\dots,p\}\), and one \(i\)-cell for each \(i\in\{q,\dots,p+q\}\). 
  The associated chain complexes with coefficients in \(\twist[s]{\ZZ}\) are given by 
  \[
    C(Q_{p,q}, \twist[s]{\ZZ}) \cong C^{(p)}(\modtwo s) \oplus C^{(p)}(\modtwo{s+q+1})[q].
  \]
\end{corollary}
This description of \(C(Q_{p,q},\twist[s]{\ZZ})\) immediately gives us additive descriptions of the homology and cohomology of \(Q_{p,q}\) with arbitrary -- untwisted and twisted -- coefficients.
For example, \(\H_\bullet(Q_{p,q},\ZZII)\) is just a direct sum of two (shifted) copies of \(\H_\bullet(\RP^p,\ZZII)\).  Table~\ref{table:split-examples} displays the cohomology groups of some split real quadrics with untwisted and twisted integral coefficients. 

\begin{corollary}\label{i-and-p-on-cohomology}
  For coefficients \(R\in \{\ZZ,\twist{\ZZ}, \ZZII\}\), the embedding \(\Delta\colon \RP^p \hookrightarrow Q_{p,q}\) and the projection \(\pi_1\colon  Q_{p,q}\twoheadrightarrow\RP^p\) induce mutually inverse isomorphisms \(\H^i(Q_{p,q}, R) \cong \H^i(\RP^p, R)\) in all degrees \(i < p\).  In degree~\(p\), \(\Delta^*\) is an epimorphism split by the monomorphism \(\pi_1^*\). 
\end{corollary}
\begin{proof}
  This is clear from Figure~\ref{fig:topology} and the fact that any split monomorphism or epimorphism \(\ZZ\to\ZZ\) or \(\ZZII\to\ZZII\) is an isomorphism.
\end{proof}

\begin{corollary}\label{mod-2-reduction}
  For \(p<q\) and \(s\in\ZZ\), each cohomology group \(\H^i(Q_{p,q},\ZZ(s))\) is isomorphic to one of the groups \(0\), \(\ZZII\) or \(\ZZ\).  In all degrees \(i\) in which \(\H^i(Q_{p,q},\twist[s]{\ZZ}) \cong \ZZ\), the mod-\(2\)-reduction map \(\H^i(Q_{p,q},\twist[s]{\ZZ})\to \H^i(Q_{p,q},\ZZII)\) is a surjection \(\ZZ\twoheadrightarrow\ZZII\).  In all degrees \(i\) in which \(\H^i(Q_{p,q},\twist[s]{\ZZ}) \cong \ZZII\), the mod-\(2\)-reduction map is an isomorphism.  
  
  The same assertions are also true in the case \(p=q\) in all degrees \(i\neq p\).  The reduction maps \(\H^p(Q_{p,p},\twist[s]{\ZZ})\to \H^p(Q_{p,p},\ZZII)\) are epimorphisms or monomorphisms as follows:
  \begin{align*}
    \ZZ\oplus\ZZ &\twoheadrightarrow \ZZII\oplus\ZZII \quad \text{when \(p+s\) is odd;}\\
    \ZZII &\hookrightarrow \ZZII\oplus\ZZII \quad \text{when \(p+s\) is even.}
  \end{align*}
\end{corollary}
\begin{proof}
  This is immediate from the cellular chain complexes. 
\end{proof}

\begin{remark}[Alternative additive computations]\label{sec:top:alt-additive}
  The additive structure of \(\H^\bullet(Q_{p,q},\ZZ\oplus\twist{\ZZ})\) can also be computed using
  localization sequences, (twisted) Thom isomorphisms and a topological version of the blow-up setup
  described in Section~\ref{sec:blow-up}.  That is, there is a topological variant
  of the algebro-geometric computations that will follow, which we leave as an exercise to the
  diligent reader.  A third approach would be to derive the topological results from the geometric ones, using the results of the forthcoming article \cite{HWXZ}, which will establish an isomorphism given by the {\em real realization functor} between the $\I$-cohomology ring of any smooth real cellular variety and the singular cohomology ring of its real points (see Theorem~\ref{HWXZ} below).  The topological computations presented here are intentionally independent of these considerations.
\end{remark}

\subsection{Mod-two cohomology ring}

\begin{proposition}\label{top:mod-2-ring}
  The cohomology ring of the real quadric \(Q_{p,q}\) with coefficients in \(\ZZII\) has the form
  \[
    \H^\bullet(Q_{p,q},\ZZII) =
    \begin{cases}
      \ZZII[\xi,\zeta]/(\xi^{p+1},\zeta^2+\xi^p\zeta) &\text{ if \(p=q\) and \(p\) is even} \\
      \ZZII[\xi,\zeta]/(\xi^{p+1},\zeta^2) &\text{ in all other cases}
    \end{cases}
  \]
  with \(\deg{\xi} = 1\) and \(\deg{\zeta} = q\).  For \(p<q\), the generators \(\xi\) and \(\zeta\) are the unique non-zero elements of the respective degrees.  In the case \(p=q\), \(\xi\) is the unique non-zero element of degree~1, and \(\zeta\) is the unique non-zero element in the kernel of \(\Delta^*\colon \H^p(Q_{p,p},\ZZII)\to \H^p(\RP^p,\ZZII)\) (cf.\  Figure~\ref{fig:topology}).
\end{proposition}

\begin{proof}[Proof of \ref{top:mod-2-ring}, apart from the computation of \(\zeta^2\) when \(p=q\)]
  Write \(h^i\) for \(\H^i(Q_{p,q},\ZZII)\).  Let \(\xi\) be the generator of \(h^1\).  By Corollary~\ref{i-and-p-on-cohomology} and the known cohomology of \(\RP^p\), each of the powers \(\xi^i\) generates \(h^i\) for \(i\in\{0,\dots,p-1\}\). Moreover, \(\xi^p\in h^p\) is non-zero, and \(\xi^{p+1} = 0\). Let \(\zeta = \zeta_0\) denote the generator of the kernel of \(\Delta^*\) on \(h^q\).  When \(p<q\), \(\ker(\Delta^*) = h^q\), so \(\zeta_0\) is simply a generator of \(h^q\). When \(p = q\), the elements \(\xi^p\) and \(\zeta_0\) together form a basis of \(h^p\).
  Pick generators \(\zeta_1, \dots, \zeta_p\) in the remaining degrees \(h^{1+q},\dots, h^{p+q}\), so that \(\deg{\zeta_i} = i+q\). Poincar\'{e} duality implies
  \(
  \xi^{p-i} \zeta_i = \zeta_p 
  \)
  for all \(i\in\{0,\dots,p\}\).
  It follows that
  \(
  \xi^i \zeta_0  = \zeta_i
  \)
  for \(i\in\{0,\dots,p\}\). In all cases with \(p<q\), it is moreover clear for degree reasons that \(\zeta^2 = 0\), so in these cases the proof is complete.
  In the case \(p=q\), it remains to compute \(\zeta^2\).  This will be done in the course of the proof of Theorem~\ref{top:integral-ring} below.
\end{proof}

\subsection{Integral cohomology ring}
The integral cohomology ring of \(\RP^p\) with twisted integral coefficients is as follows:
\begin{align}\label{eq:H(RP^p)}
  \H^\bullet(\RP^{p},\ZZ\oplus\twist{\ZZ}) 
  &= \ZZ[\xi,\alpha]/(2\xi, \xi^{p+1},\xi\alpha,\alpha^2)
    \quad\text{ with }
    \begin{cases}
      \deg{\xi} = \twistedbideg{1}{1}\\
      \deg{\alpha} = \twistedbideg{p}{p+1}
    \end{cases}
\end{align}
Mod-\(2\)-reduction to
\(  
\H^\bullet(\RP^{p},\ZZII) = \ZZII[\xi]/\xi^{p+1}
\)
is determined by \(\xi\mapsto\xi\) and \(\alpha\mapsto \xi^p\).  

\begin{theorem}\label{top:integral-ring}
  The integral cohomology ring of the real quadric \(Q_{p,q}\) has the following form:
  \begin{align*}
    \H^\bullet(Q_{p,q},\ZZ\oplus\twist{\ZZ}) 
    &\cong 
      \begin{cases}
        \ZZ[\xi,\alpha,\beta]/(2\xi,\xi^{p+1},\xi\alpha,\alpha^2,\beta^2-\alpha\beta) & \text{ if \(p = q\) and \(p\) is even} \\
        \ZZ[\xi,\alpha,\beta]/(2\xi,\xi^{p+1},\xi\alpha,\alpha^2,\beta^2) & \text{ in all other cases}
      \end{cases}
  \end{align*}
  with generators of degrees
  \(\deg{\xi}    = \twistedbideg{1}{1}\),
  \(\deg{\alpha} = \twistedbideg{p}{p+1}\),
  \(\deg{\beta}  = \twistedbideg{q}{q+1}\).
  The generators \(\xi\) and \(\alpha\) are pullbacks under \(\pi_1\) (cf.\ Figure~\ref{fig:topology}) of the generators of \(\H^\bullet(\RP^p,\ZZ\oplus\ZZ(1))\) that have the same names in~\eqref{eq:H(RP^p)}.
  Under mod-\(2\)-reduction, \(\xi\mapsto \xi\), \(\alpha\mapsto \xi^p\) and \(\beta\mapsto \zeta\).
  See Table~\ref{table:split-examples} for some examples.
\end{theorem}

\begin{table}
  \newcommand*{\h}{\operatorname{h}}
  \newcommand*{\zz}{\mathrm{z}\hspace{-0.6ex}\mathrm{z}}
  \newcommand*{\two}[2]{\oplus\;\raisebox{-0.2ex}{\parbox{3.5em}{\(#1\)\\[-3pt]\(#2\)}}}
  \begin{adjustwidth}{-5em}{-5em}
    \begin{center}
      \newcolumntype{L}{>{\(}l<{\)}}
      \newcolumntype{H}{>{\centering\arraybackslash\(}p{3.5em}<{\)}}
      \setlength{\tabcolsep}{0pt}
      \renewcommand{\arraystretch}{0.8}
      \begin{tabular}{LHHHHHHHHHHHHL} 
        \toprule
        & {{}_0} & {{}_1} & {{}_2}   & {{}_3}   & {{}_4}                    & {{}_5}                    & {{}_6}        & {{}_7}        & {{}_8}         & {{}_9}        & {{}_{10}}      & {{}_{11}} &                         \\[6pt]
        \midrule
        \H^\bullet(Q_{4,4},\ZZ)              & \ZZ 1  & 0      & \zz\xi^2 & 0        & \zz\xi^4                  & \zz\xi\beta               & 0             & \zz\xi^3\beta & \ZZ\alpha\beta &               &                &           & (\beta^2 = \alpha\beta) \\[3pt]
        \H^\bullet_-(Q_{4,4},\twist[1]{\ZZ}) & 0      & \zz\xi & 0        & \zz\xi^3 & \two{\ZZ\alpha}{\ZZ\beta} & 0                         & \zz\xi^2\beta & 0             & \zz\xi^4\beta                                                                         \\[8pt]
        \h^\bullet(Q_{4,4})                  & \zz 1  & \zz\xi & \zz\xi^2 & \zz\xi^3 & \two{\zz\xi^4}{\zz\zeta}  & \zz\xi\zeta               & \zz\xi^2\zeta & \zz\xi^3\zeta & \zz\xi^4\zeta  &               &                &           & (\zeta^2 = \xi^4\zeta)  \\[8pt]
        \midrule
        \H^\bullet(Q_{4,5},\ZZ)              & \ZZ 1  & 0      & \zz\xi^2 & 0        & \zz\xi^4                  & \ZZ\beta                  & 0             & \zz\xi^2\beta & 0              & \zz\xi^4\beta                                                        \\[3pt]
        \H^\bullet(Q_{4,5},\twist[1]{\ZZ})   & 0      & \zz\xi & 0        & \zz\xi^3 & \ZZ\alpha                 & 0                         & \zz\xi\beta   & 0             & \zz\xi^3\beta  & \ZZ\alpha\beta                                                       \\[8pt]
        \h^\bullet(Q_{4,5})                  & \zz 1  & \zz\xi & \zz\xi^2 & \zz\xi^3 & \zz\xi^4                  & \zz\zeta                  & \zz\xi\zeta   & \zz\xi^2\zeta & \zz\xi^3\zeta  & \zz\xi^4\zeta &                &           &                         \\[8pt]
        \midrule
        \H^\bullet(Q_{5,5},\ZZ)              & \ZZ 1  & 0      & \zz\xi^2 & 0        & \zz\xi^4                  & \two{\ZZ\alpha}{\ZZ\beta} & 0             & \zz\xi^2\beta & 0              & \zz\xi^4\beta & \ZZ\alpha\beta &           & (\beta^2 = 0)           \\[3pt]
        \H^\bullet(Q_{5,5},\twist[1]{\ZZ})   & 0      & \zz\xi & 0        & \zz\xi^3 & 0                         & \zz\xi^5                  & \zz\xi\beta   & 0             & \zz\xi^3\beta  & 0             & \zz\xi^5\beta                                        \\[8pt]
        \h^\bullet(Q_{5,5})                  & \zz 1  & \zz\xi & \zz\xi^2 & \zz\xi^3 & \zz\xi^4                  & \two{\zz\xi^5}{\zz\zeta}  & \zz\xi\zeta   & \zz\xi^2\zeta & \zz\xi^3\zeta  & \zz\xi^4\zeta & \zz\xi^5\zeta  &           & (\zeta^2 = 0)           \\[8pt]
        \midrule
        \H^\bullet(Q_{5,6},\ZZ)              & \ZZ 1  & 0      & \zz\xi^2 & 0        & \zz\xi^4                  & \ZZ\alpha                 & 0             & \zz\xi\beta   & 0              & \zz\xi^3\beta & 0              & \zz\xi^4\beta                       \\[3pt]
        \H^\bullet(Q_{5,6},\twist[1]{\ZZ})   & 0      & \zz\xi & 0        & \zz\xi^3 & 0                         & \zz\xi^5                  & \ZZ\beta      & 0             & \zz\xi^2\beta  & 0             & \zz\xi^4\beta  & \ZZ\alpha\beta                      \\[8pt]
        \h^\bullet(Q_{5,6})                  & \zz 1  & \zz\xi & \zz\xi^2 & \zz\xi^3 & \zz\xi^4                  & \zz\xi^5                  & \zz\zeta      & \zz\xi\zeta   & \zz\xi^2\zeta  & \zz\xi^3\zeta & \zz\xi^4\zeta  & \zz\xi^5  &                         \\[2pt]
        \bottomrule
      \end{tabular}
    \end{center}
  \end{adjustwidth}
  \vspace{1ex}
  \caption{%
    The integral cohomology and the mod-\(2\)-cohomology \(\h^\bullet := \H^\bullet(-,\ZZII)\) of some real split quadrics, with \(\zz\) denoting \(\ZZII\). 
  }\label{table:split-examples}
\end{table}

\begin{proof}[Proof of Theorem~\ref{top:integral-ring}, Part~I (integral products and reduction formulas)]
  It is easy to verify that the quotient rings \(\ZZ[\xi,\alpha,\beta]/(\dots)\) displayed in the theorem have the correct additive structure.  To verify that the ring structure is correct, we begin by computing the products of the non-torsion classes (\(\alpha\) and \(\beta\)).

  We first consider the cases with \(p < q\). In these cases, non-torsion classes \(\alpha\) and \(\beta\) of the degrees specified in Theorem~\ref{top:integral-ring} are unique up to signs, and we already know from the additive structure that \(\H^i(Q_{p,q},\ZZ(s))=0\) in the bidegrees \((i,s)=(2p,\bar 0)\) and \((i,s)=(2q,\bar 0)\) in which \(\alpha^2\) and \(\beta^2\) reside.  So \(\alpha^2=\beta^2=0\). It now follows from (the twisted version of) Poincar\'{e} duality (see \cite[Theorem~5.7 and the following remarks]{davis-kirk}) that the product \(\alpha\beta\) is a generator.  
  
  The case \(p=q\) requires more thought. We need to understand the product
  \[
    \H^p(Q_{p,p},\twist[p+1]{\ZZ}) \times \H^p(Q_{p,p},\twist[p+1]{\ZZ}) \to \H^{2p}(Q_{p,p},\ZZ),
  \]
  where \(\H^{p}(Q_{p,p},\twist[p+1]{\ZZ})=\ZZ\oplus \ZZ\).  Recall from Figure~\ref{fig:topology} our notation \(\pi\) for the two-fold cover \(\pi\colon S^p\times S^p\to Q_{p,p}\), the notation \(\pi_1,\pi_2\colon Q_{p,p}\rightrightarrows \RP^p\) for the compositions of the two-fold cover \(Q_{p,p}\to \RP^p\times\RP^p\) with the projections onto the two factors, and \(\pi^S_1,\pi^S_2\colon S^p\times S^p\rightrightarrows S^p\) for the projections onto the two sphere factors.
  Consider the two-fold cover \(S^p\to \RP^p\).  The induced map on the top cohomology group, \(\H^p(S^p,\ZZ)\leftarrow \H^p(\RP^p,\ZZ(p+1))\), is multiplication by \(\pm 2\), as can be seen either explicitly from the cellular computations or from degree considerations.
  We may therefore choose generators \(\sigma\in \H^p(S^p,\ZZ)\) and \(\alpha\in \H^p(\RP^p,\twist[p+1]{\ZZ})\) such that \(\alpha\) maps to \(2\sigma\) under this pullback map.  Define \(\alpha_i := \pi_i^*(\alpha)\) and \(\sigma_i := (\pi_i^S)^*\sigma\), so that
  \begin{equation*}
    \pi^*(\alpha_i) = 2\sigma_i.
  \end{equation*}
  Choose a generator \(\beta\) of \(\ker(\Delta^*)\subset \H^{p}(Q_{p,p},\twist[p+1]{\ZZ})\). As \(\pi_1^*\) is split by \(\Delta^*\), it follows that \(\alpha_1\), \(\beta\) form a \(\ZZ\)-basis of \(\H^p(Q_{p,p},\twist[p+1]{\ZZ})\). The pullback of \(\beta\) lives in the kernel of \((\Delta^S)^*\), hence can be written as 
  \begin{equation*}
    \pi^*\beta = a(\sigma_1 - \sigma_2)
  \end{equation*}
  for some \(a \in \ZZ\).

  We now analyse what happens in cohomological degree~\(2p\). 
  As the generators \(\alpha_i\) are pulled back from \(\RP^p\), we find that \(\alpha_1^2 = \alpha_2^2=0\).   In  \(\H^\bullet(S^p\times S^p,\ZZ)\), we know that similarly \(\sigma_1^2 = \sigma_2^2 = 0\), while \(\sigma_1\sigma_2\) is a generator in degree \(2p\).
  The pullback \(\pi^*\colon \H^{2p}(S^p\times S^p,\ZZ)\leftarrow \H^{2p}(Q_{p,p},\ZZ)\) is multiplication by \(\pm 2\).
  (Indeed, degree considerations show that \(\pi^*\) is either \(\pm 2\) or zero, and by composing with the transfer we find that \(\pi^*\) is either \(\pm 2\) or \(\pm 1\).)
  We may therefore choose a generator \(\gamma\in \H^{2p}(Q_{p,p},\ZZ)\) such that \(\pi^*\gamma = 2\sigma_1\sigma_2\).  For this generator, we find \(\alpha_1\alpha_2 = 2\gamma\).
  For the remaining products of \(\alpha_1\) and \(\beta\), we obtain:
  \begin{align*}
    \pi^*(\beta^2) 
    & =  a^2 (-\sigma_1\sigma_2 - (-1)^{p^2}\sigma_1\sigma_2) =  -a^2\tfrac{1}{2}(1+(-1)^p)\pi^*\gamma, 
    &  & \text{ so} & \beta^2       & = -a^2\tfrac{1}{2}(1+(-1)^p)\gamma, 
    \\
    \pi^*(\alpha_1\beta) 
    & = 2a \sigma_1 (\sigma_1-\sigma_2) = -2a \sigma_1\sigma_2 = -a\pi^*\gamma,
    &  & \text{ so} & \alpha_1\beta & = -a\gamma
  \end{align*}
  Thus, in the basis of \(\H^p(Q_{p,p},\twist[p+1]{\ZZ})\) given by \(\alpha_1, \beta\), the product is described by the following matrix:
  \[\begin{pmatrix}
      0  & -(-1)^pa \\
      -a & -a^2\tfrac{1}{2}(1+(-1)^p)
    \end{pmatrix}\]
  By Poincar\'{e} duality, this matrix must define a perfect pairing on \(\ZZ^2\). So \(a=\pm 1\), and by changing the sign of \(\beta\) if necessary, we may as well assume that \(a=1\).  It follows that \(\alpha_1\beta\) is a generator of \(\H^{2p}(Q_{p,p},\ZZ)\).  For even \(p\), we moreover find that \(\beta^2 = \alpha_1\beta\);  for odd \(p\), we find that \(\beta^2 = 0\). Altogether, this is precisely the result displayed above, with \(\alpha_1\) written as \(\alpha\). 
  (We also find that \(2\beta = \alpha_1-\alpha_2\).) 
  
  We now determine the images of the various generators under mod-\(2\)-reduction.
  As the elements \(\xi\) and \(\alpha\) are pulled back from cohomology classes of \(\RP^p\) under \(\pi_1\), the formulas for these elements follow from the corresponding formulas for \(\RP^p\).  The elements \(\beta\in \H^p(Q_{p,p},\ZZ(p+1))\) and \(\zeta\in \H^p(Q_{p,p},\ZZII)\) were both chosen as generators of the kernel of the map induced by \(\Delta\colon \RP^p\hookrightarrow Q_{p,p}\). As \(\Delta\) is split by \(\pi_1\), it follows that \(\beta\mapsto \zeta\), as claimed.
\end{proof}
\begin{proof}[End of proof of Proposition~\ref{top:mod-2-ring}]
  We have just observed that \(\beta\) reduces to \(\zeta\).  The formula for \(\zeta^2\) displayed in Proposition~\ref{top:mod-2-ring} is therefore immediate from the formula for \(\beta^2\) that we have already verified.
\end{proof}

\begin{proof}[Proof of Theorem~\ref{top:integral-ring}, Part~II (torsion products)]
  As observed in Corollary~\ref{mod-2-reduction}, each cohomology group \(\H^i(Q_{p,q},\ZZ(s))\) is either isomorphic to \(\ZZII\) or free abelian.  For the proof of Theorem~\ref{top:integral-ring}, it remains to compute the products of two-torsion cohomology classes with arbitrary cohomology classes.  So take a two-torsion class \(\mu\in\H^i(Q_{p,q},\ZZ(s))\) and an arbitrary class \(\mu'\in\H^{i'}(Q_{p,q},\ZZ(s'))\). Then the product \(\mu\mu'\in\H^{i+i'}(Q_{p,q},\ZZ(s+s'))\) is again a two-torsion class.  If \(\H^{i+i'}(Q_{p,q},\ZZ(s+s'))\) is free abelian, we deduce \(\mu\mu' = 0\).  If, on the other hand, \(\H^{i+i'}(Q_{p,q},\ZZ(s+s'))\cong \ZZII\), then by Corollary~\ref{mod-2-reduction} the reduction map \(\H^{i+i'}(Q_{p,q},\ZZ(s+s'))\to \H^{i+i'}(Q_{p,q},\ZZII)\) is injective. We can therefore compute the product \(\mu\mu'\) by passing to the mod-\(2\)-reductions \(\bar\mu\) and \(\bar\mu'\) and computing the product \(\bar\mu\bar\mu'\) in \(\H^\bullet(Q_{p,q},\ZZII)\), whose ring structure we already know.
\end{proof}

\section{The blow-up setup of Balmer-Calm\`es}\label{sec:blow-up}
We now turn to algebraic geometry.   Recall from Section~\ref{sec:notation} that \(Q_n\) denotes the split \(n\)-dimensional quadric over a smooth scheme \(S\) over a field \(F\),
where \(n \geq 3\). We begin by summarizing some material from \cite{Ne}.

\begin{lemma}\label{lem:normal-bundles-in-Q}
  Let \(\sheaf N_{\PPpX}\) denote the normal bundle of \(\PPpX\) in \(Q_n\).  
  Then \(\det(\sheaf N_{\PPpX})\cong \Oo_{\PPpX}(n-p-1)\).
\end{lemma}
\begin{proof}
  See \cite[Lemma~6.1 and above Theorem~6.4]{Ne}.
\end{proof}

\begin{definition}\label{def:affinebundle}
  An affine space bundle of rank \(r\) is a Zariski locally trivial fibre bundle \(f\colon X \rightarrow Y\) with fibres isomorphic to \(\AA^r\). This means that \(Y = \cup_{i \in I} U_i\) can be covered by open subschemes \(U_i\) such that \(f^{-1}(U_i) \cong U_i \times \AA^r\) over \(U_i\) for every \(i \in I\). 
\end{definition}

\begin{lemma}\label{lem:rho-is-vb}
  The morphism \(\rho\colon Q_n - \PPpX \to \PPpY\) described in Section~\ref{sec:notation} is an affine space bundle of rank \(p+1\) with \(i_y\colon \PPpY \hookrightarrow Q_n - \PPpX\) as a global section.  When \(n\) is even, \(\rho\) is even a vector bundle, and \(i_y\) is its zero section. 
\end{lemma}

\begin{proof}
  This seems to be well-known (see \cite[Section~6]{Ne}).  As we are not aware of a suitable reference, we provide here a proof for the case of odd \(n\) for the reader's convenience. The case of even \(n\) is similar.  Let \(U\cong \AA^p\) be the open subscheme of \(\PPpY\) defined by \(y_i \neq  0\). We will check that 
  \(
  \rho^{-1}(U)\cong U \times \AA^{p+1}
  \).
  We may assume that \(S = \Spec(F)\), as all schemes are already defined over~\(F\).  Let us moreover assume for ease of notation that \(i = 0\).   Then \(\rho^{-1}(U)\) is an affine quadric in \(\AA^{2p+2} \cong \PP^{2p+2} - V(y_0)\), defined in terms of the coordinates \([ x_0: \ldots : x_p : 1:  y_1 : \ldots : y_p:z]\) by the equation \( x_0 + x_1 y_1 + \cdots + x_p y_p + z^2 = 0\).  This affine quadric  is isomorphic to \(U\times \AA^{p+1}\) via an isomorphism of coordinate rings as follows:
  \begin{align*}
    \frac{F[x_0, x_1, \ldots , x_p, y_1, \ldots , y_p , z] }{ (x_0 + x_1 y_1 + \cdots + x_p y_p + z^2) } &\rightarrow F[ x_1, \ldots , x_p, y_1, \ldots , y_p , z] \\
    x_0 & \mapsto -(x_1 y_1 + \cdots + x_p y_p + z^2)\\
    x_i & \mapsto x_i \quad \text{ for } i\neq 0\\
    y_i & \mapsto y_i \quad \text{ for all } i
  \end{align*}
  This isomorphism is clearly compatible with the projections to~\(U\).
\end{proof}

\begin{lemma}\label{lem:picquadrics}
  Let \(n \geq 3\). Let \({\iota}_w\colon  \PP^p_w \to Q_n\) denote one of the closed embeddings \(\iota_x\) or \(\iota_y\) (or \(\iota_{x'}\) or \(\iota_{y'}\) if \(n\) is even).  This morphism induces an isomorphism \(\iota^*_w\colon \Pic(Q_n) \stackrel{\cong}{\to}\Pic(\PP_w^p)\) under which \(\iota_w^*\Oo_Q(k) \cong \Oo_{\PP^p_w}(k)\).  In particular, each of the mentioned embeddings induces the \emph{same} isomorphism on Picard groups.
\end{lemma}
\begin{proof}
  We only deal with \(\PPpY\) here, as the other cases are analogous. Recall that \(\Pic(X)\cong \CH^{1}(X)\) when \(X\) is smooth over a field.
  Consider the pullback map induced by the morphism along the top row of Figure~\ref{fig:geometry}:
  \[
    \ZZ \oplus \Pic(S) \cong \CH^1(\PPpY) \xleftarrow{i_y^*} \CH^1(Q_n-\PPpX) \xleftarrow{j^*} \CH^1(Q_n) \xleftarrow{i^*} \CH^1(\PP^{n+1}) \cong \ZZ \oplus \Pic(S)
  \]
  Using Lemma~\ref{lem:rho-is-vb} and the homotopy invariance of Chow groups, we see that \(i_y^*\)  is an isomorphism. Using the assumption \(n \geq 3\) and the localization sequence of Chow groups, we see that \(j^*\) is an isomorphism. As the composition \(i\circ j \circ i_y\) is a linear embedding of \(\PPpY\) into \(\PP^{n+1}\), the composition of all these pullback maps is likewise an isomorphism, sending \(\Oo_{\PP^{n+1}}(1)\) to \(\Oo_{\PPpY}(1)\).  
  The claim follows.
\end{proof}

The closed immersion \(\iota\colon \PP^{p}_{x} \hookrightarrow Q_n\) is a regular immersion of codimension \(p\) (if \(n=2p\)) or \(p+1\) (if \(n=2p+1\)), since  \(\PP^{p}_{x}\) and \(Q_n\) are both smooth (cf.\  \cite[Ch.\,4, 17.12.1]{EGA4}). Let \(Bl\) denote the blow-up of \(Q_n\) along \(\PP^{p}_{x}\) and let \(E\) denote the exceptional fibre. The following proposition shows that this setup satisfies \cite[Hypothesis 1.2]{BC}.

\begin{proposition}\label{pblup}
  There exists a morphism  \(\tilde{\rho}\colon  Bl \rightarrow \PP^{p}_{y}\) making the following diagram commutative:
  \begin{equation}\label{bc}
    \begin{aligned}
      \xymatrix{
        \PP^{p}_{x} \ar[r]^-{\iota} 
        &Q_n   & \ar@{_{(}->}[l]^-{j} U := Q_n - \PP^{p}_{x}  \cong Bl - E \ar@{.>}[d]^-\rho   \ar[dl]_-{\tilde{j}}   \\
        E \ar[u]^-{\tilde{\pi}}  \ar@{^{(}->}[r]^-{\tilde{\iota}} &Bl \ar@{.>}[r]^-{\tilde{\rho}} \ar[u]^-{\pi} & \PP^{p}_{y}  }
    \end{aligned}
  \end{equation}
  Here, \(\iota\) and \(\tilde{\iota}\) are closed immersions, \(\pi\) and \(\tilde{\pi}\) are projections, \(j\) and \(\tilde{j}\) are open immersions, and \(\rho\) is the affine space bundle from Lemma~\ref{lem:rho-is-vb}.
\end{proposition}

\begin{proof}
  We concentrate on the case \(n=2p\); the case of odd \(n\) is similar. Then \(Bl\) is contained in (in fact equal to) the closed subscheme of 
  \[
    Q_n \times \PP^p = \Proj(\Oo_S[x_0,\ldots,x_p,y_0,\ldots,y_p]/(x_0y_0+ \dots +x_py_p)) \times \Proj(\Oo_S[T_{0},\dots,T_{p}])
  \] 
  defined by the homogeneous polynomials  \(y_{k}T_{j} - y_{j}T_{k}\) (for \(k,j = 0,1,\dots,p\)) and \(\sum_{i=0}^{p} x_i T_i = 0\).  This follows from the universal property of a blow-up, or from the geometric description of the blow-up along \(\PPpX\) as the closure of the graph of the rational map \(Q_n \to \PP^d, [x_0:\ldots:x_p:y_0:\ldots:y_p]\mapsto [y_0:\ldots:y_p]\), see \cite[paragraph above Exercise~7.19]{harris}.  Define \(\tilde{\rho}\) to be the composition \(Bl \hookrightarrow Q_n \times \PP^{p}\rightarrow \PP^{p}_{y}\), where the first map is the inclusion and the last map is the projection onto the  second factor.  The relations \(y_{k}T_{j} = y_{j}T_{k}\) guarantee the commutativity of diagram~\eqref{bc}.  
\end{proof}   

Diagram~\eqref{bc} descends to the following diagram of Picard groups:
\begin{equation}\label{pic:iso}
  \begin{aligned}
    \xymatrix{
      &\Pic(Q_n) \ar[d]^{\pi^{*}}_{\binom{1}{0}}  \ar[r]^{j^{*} \simeq}& \Pic(U)\ar[d]^{(\rho^{*})^{-1}\simeq} \\
      &	\Pic(Q_n)\oplus \mathbb{Z}[E] \cong  \Pic(Bl) \ar[ru]^-{\tilde{j}^*}& \ar[l]^-{\tilde{\rho}^{*}}_-{\binom{1}{\lambda}} \Pic(\PP^{p}_{y})
    }
  \end{aligned}
\end{equation}
The map \(\lambda\colon\Pic(\PP^p_y) \rightarrow \ZZ[E]\) does not vanish in general. Hence, the square does \emph{not} generally commute (cf.\  \cite[Remark 2.2]{BC}), but both triangles in this diagram commute.  For even \(n\), we now compute the value of \(\lambda(\Oo(1))\) as \(-1\):

\begin{proposition}\label{prop:lambda}
  Assume \(n=2p\) with \(p \geq 2\). The pullback homomorphism  \(\tilde{\rho} = \binom{1}{\lambda}\colon \Pic(\PP^p_y)  \rightarrow \Pic(Q_n) \oplus \mathbb{Z}[E]\) sends \(\Oo_{\PP^p}(1)\) to \((\Oo_Q(1) , -1)\).
\end{proposition}
\begin{proof} 
  All schemes in sight are smooth over a field, so we can identify all Picard groups with codimension one Chow groups as above. The dimensions of our schemes are \(\dim(Bl) = \dim (Q_n) = 2p\) and \(\dim(E) = 2p - 1\).  The identification of \(\CH^1(Bl)\) with \(\CH^1(Q_n)\oplus \ZZ[E]\) is given by \(\pi^*\colon \CH^1(Q_n)\to \CH^1(Bl)\) and by \(\tilde{\iota}_*\colon \CH^0(E) = \ZZ[E]\to \CH^1(Bl)\).  
  This identification is obtained by noting that the usual localization sequence of Chow groups is split exact:
  \begin{equation}\label{eq:ch}
    0 \to \CH^{0}(E)  \xrightarrow{\tilde{\iota}_{*}} \CH^{1}(Bl) \xrightarrow{\tilde{j}^{*}}  \CH^{1}(U) \to 0
  \end{equation}
  To see this, recall that for smooth schemes this exact sequence fits into a long exact sequence of motivic cohomology (cf.\  \cite[\S\,3]{DI:Hopf}):
  \[ 
    \cdots \to \H^{1,1}(Bl) \xrightarrow{\tilde{j}^{*}}   \H^{1,1}(U) \xrightarrow{\partial} \CH^{0}(E) \xrightarrow{\tilde{\iota}_{*}} \CH^{1}(Bl) \xrightarrow{\tilde{j}^{*}}  \CH^{1}(U) \rightarrow 0    
  \]
  The sequences breaks down into short split exact sequences because \(\tilde{j}^*\) is split surjective via \(\tilde{\rho}^{*} \circ (\rho^{*})^{-1}\) for all degrees.  The map \(\tilde{j}^{*}\colon  \CH^{1}(Bl) \rightarrow  \CH^{1}(U)\) in degree one can also be split by \( \pi^{*} \circ (j^{*})^{-1}\).
  
  Next, we describe the maps in diagram~\eqref{pic:iso} explicitly in terms of generators.  Let us write \([y_0]\in \CH^1(Q_n)\) for the cycle on \(Q_n\) corresponding to the subscheme defined by \(y_0 = 0\), and analogously for cycles on other schemes.  In this notation, each of the groups \(\CH^1(\PP^p_y)\), \(\CH^1(U)\) and \(\CH^1(Q_n)\) is generated by the cycle \([y_0]\) corresponding to the line bundle \(\Oo(1)\) in each Picard group.  The pullback maps are determined by  \(\rho^*[y_0] = [y_0] = j^*[y_0]\), \(\tilde{\rho}^*[y_0] =  [T_0]\) and \(\pi^*[y_0] = [y_0]\).  By the commutativity of both triangles in diagram~\eqref{pic:iso}, we see \(\tilde{j}^*[T_0] = [y_0] =  \tilde{j}^*[y_0]\) in \(\CH^1(U)\).  So the exact sequence \ref{eq:ch} shows that 
  \begin{equation}\label{eq:prop:lambda:1}
    [T_0] - [y_0] = \lambda [E]
  \end{equation}
  in \(\CH^1(Bl)\) for some integer \(\lambda\). This is the integer that we need to compute. 
  
  To compute \(\lambda\), first note that for the closed subschemes \(V(y_0)\) and \(V(T_0)\) of \(Bl\) we have an equality of sets \(V(y_0) = V(T_0) \cup E\) with neither of \(E\) or \(V(T_0)\) contained in one another.   Here, the exceptional divisor \(E\) is the smooth subscheme defined by \(\{( [x_0:\ldots:x_p],[T_{0}:\ldots:T_{p}]) \in\PP^{p} \times \PP^{p}: \sum x_iT_{i} = 0 \}\). In particular, \(E\) is integral, hence an irreducible component of \(V(y_0)\).   Using \cite[Section~1.5]{Fu}, we conclude 
  \begin{equation}\label{eq:prop:lambda:2}
    [y_0] = [T_0] + \ell [E],
  \end{equation}
  where \(\ell\) is the length of \(O_{V(y_0),E}\) as a module over itself.  Now note that \(\Oo_{V(y_0),E}\) is a field: \(\Oo_{Bl,E}\) is a discrete valuation ring with maximal ideal given by \((y_0, y_1, \cdots, y_p)\); as \(y_i = y_0 \frac{T_p}{T_0}\), this ideal is principal, generated by \(y_0\).  The relation \(y_0 =0\) in \(\Oo_{V(y_0),E}\) kills this maximal ideal of \(\Oo_{Bl,E}\), so \(\Oo_{V(y_0),E}\) is a field as claimed.  It follows that \(\ell = 1\).

  By comparing \eqref{eq:prop:lambda:1} and \eqref{eq:prop:lambda:2}, we conclude that \(\lambda = -1\).
\end{proof}

\section{$\I$-cohomology: additive structure}\label{additivecomp}

We now embark on our computations of $\I$-cohomology of split quadrics, keeping the notation established in Section~\ref{sec:notation} and in the previous section. For the definition and basic properties of $\I$- and \(\Ib\)-cohomology, we refer to \cite{faselthesis}, \cite{MField}, \cite{AF} and the survey lectures of Jean Fasel in these Proceedings \cite{fasel:this-volume}.  In particular, these groups may be described using a variant of the Gersten complex, in which the entries are powers of the fundamental ideal in the Witt ring of the appropriate residue fields.  We also refer to \cite{faselthesis} for twisted coefficients and how they appear when studying pushforwards, and to Lemmas~\ref{lem:normal-bundles-in-Q} and \ref{lem:picquadrics} for possible twists that appear for split quadrics. Finally, we note that similarly to Chow groups, \(\I\)-cohomology groups satisfy homotopy invariance for affine space bundles over smooth bases:

\begin{theorem}\label{thm:homotopy-invariance}
  Let \(f \colon X \to Y\) be an affine space bundle (recall Definition~\ref{def:affinebundle}) over a smooth variety \(Y\) over a base field \(F\) of characteristic \(\neq 2\).
  Then \(f^*\colon \H^i(Y, \I^j,  \LC) \rightarrow \H^i(X, \I^j, f^* \LC) \) is an isomorphism for any line bundle \(\LC\) on \(Y\). 
\end{theorem}
\begin{proof}
  Since \(Y\) is quasicompact, we may assume that the open cover in Definition~\ref{def:affinebundle} is a finite cover by open affines. The pullback along the restriction of \(f\) to any affine subset is an isomorphism by \cite[Corollaire~11.2.8]{faselthesis}.  Arguing by induction, we are reduced to the following commutative ladder diagram with \(U\) affine, in which the rows are exact Mayer-Vietoris sequences:
  \[ \xymatrix { \cdots  \ar[r] & \H^i(Y, \I^j, \LC) \ar[d]_-{f^*} \ar[r] & \H^i(U, \I^j, \LC|_U) \oplus \H^i(V, \I^j, \LC|_V)  \ar[d]^-{(f^*_U, f^*_V)} \ar[r] & \H^i(U \cap V, \I^j, \LC|_{U\cap V}) \ar[d]^-{f_{U\cap V}^*}\ar[r]& \cdots \\
      \cdots  \ar[r] & \H^i(X, \I^j, f^*\LC) \ar[r] & \H^i(X_U, \I^j, f^*(\LC|_U)) \oplus \H^i(X_V, \I^j, f^*(\LC|_V))  \ar[r] & \H^i(X_{U\cap V}, \I^j, f^*(\LC|_{U\cap V})) \ar[r]& \cdots}  \]
  The claim now follows by the five lemma. 

  Note that the required Mayer-Vietoris sequence can be 
  deduced from the localization sequence \cite[Th\'{e}or\`{e}me 9.3.4]{faselthesis} and  excision along open embeddings for \(\I\)-cohomology.  (More generally, we have excision along flat morphisms, cf.\  \cite[Lemma~3.7]{calmesfasel}.)
\end{proof}  
\begin{theorem}[Base change formula {\cite[Theorem~2.12]{AF}}]\label{thm:base-change-formula}
  Suppose that \(f\colon X \rightarrow Y\) is a regular codimension~\(c\) embedding of smooth schemes that fits into a cartesian diagram of smooth schemes as follows:
  \[\xymatrix{
      X' \ar[d]^-{g}\ar[r]^-{v} & X \ar[d]^-{f} \\
      Y'  \ar[r]^-{u} & Y
    }
  \]
  Suppose that the natural map on normal bundles \(\Omega\colon  \mathcal{N}_{X'}Y' \rightarrow v^*\sheaf N_{X}Y\) induced by \(u\) and \(v\) is an isomorphism. Then \(u^* \circ f_* =  g_* \circ  \det \Omega^\vee  \circ v^*\), i.e.\ the following square commutes for any line bundle \(\sheaf L\) over \(Y\):
  \[\xymatrix{
      \ar[d]_-{\det \Omega^\vee}^{\cong}
      \H^i(X',\I^j,v^*f^*\sheaf L\otimes v^*\det(\sheaf N_XY)^\vee)
      &
      \ar[l]_-{v^*}
      \H^i(X,\I^j,f^*\sheaf L\otimes \det(\sheaf N_XY)^\vee)
      \ar[dd]_{f_*}
      \\
      \H^i(X',\I^j,g^*u^*\sheaf L\otimes \det(\sheaf N_{X'}{Y'})^\vee)
      \ar[d]_{g_*}
      \\
      \H^{i+c}(Y',\I^{j+c},u^*\sheaf L)
      &
      \ar[l]_-{u^*}
      \H^{i+c}(Y,\I^{j+c},\sheaf L)
    }
  \]
\end{theorem}

\subsection{Fasel's computations for projective spaces}\label{sec:Faselproj}
Let  \(S\)  be a smooth scheme over a field \(F\) of characteristic \(\neq 2\). The following isomorphisms of graded abelian groups are proved by Fasel in \cite{fasel:ij}: see Corollary~5.8, Definition~5.9, and Theorems 9.1, 9.2 and 9.4 of loc.\ cit.
(Note that the sequence in 9.4 of loc.\ cit.\ splits as the bundle \(E\) Fasel considers is trivial in our case.) 
\begin{theorem}[Fasel]\label{thm:Faselproj}
  \[
    \H^i( \PP^p, \I^j, \Oo(l) ) \cong
    \left\{
      \begin{array}{ll}
        \big(\bigoplus\limits_{\substack{m \textnormal{ even}\\2\leq m \leq p}} \H^{i-m}(S, \Ib^{j-m}) \big) \oplus \H^{i}(S,\I^j)  & \textnormal{ if \(l\) is even and \(p\) is even} \\[18pt]
        \big(\bigoplus\limits_{\substack{m \textnormal{ even}\\ 2\leq m \leq p}} \H^{i-m}(S, \Ib^{j-m}) \big) \oplus \H^{i}(S,\I^j)	\oplus \H^{i-p}(S,\I^{j-p}) & \textnormal{ if \(l\) is even and \(p\) is odd} \\[18pt]
        \big(\bigoplus\limits_{\substack{m \textnormal{ odd}\\ 1\leq m \leq p}} \H^{i-m}(S, \Ib^{j-m}) \big) \oplus \H^{i-p}(S,\I^{j-p})   & \textnormal{ if \(l\) is odd and \(p\) is even} \\[18pt]
        \bigoplus\limits_{\substack{m \textnormal{ odd}\\ 1\leq m \leq p}} \H^{i-m}(S, \Ib^{j-m})   & \textnormal{ if \(l\) is odd and \(p\) is odd} \\
      \end{array} \right. \]
\end{theorem}

Let \(p\colon \PP^p \rightarrow S\) be the projection map and \(s\colon S\rightarrow \PP^p\) a rational point. The isomorphism above on the component \( \H^{i-m}(S, \I^{j-m}) \rightarrow \H^i(\PP^p, \I^j, \Oo(l))\) is given by the pullback \(p^*\) if \(l =m =0\) and by the pushforward \(s_*\) if \(m = p\) and \(l = p-1\). The map 
\begin{equation}\label{eq:defn:mu_m}
  \mu^m_\LC\colon \H^{i-m}(S, \Ib^{j-m}) \rightarrow \H^i(\PP^p, \I^j, \LC(-m)) 
\end{equation}
in the isomorphism above is defined in \cite[Definition~5.1]{fasel:ij}; explicitly it is the composition 
\[  
  \xymatrix{
    \H^{i-m}(S, \Ib^{j-m})  \ar[r]^-{p^*} &\H^{i-m}(\PP^p, \Ib^{j-m}) \ar[r]^-{\partial_{\LC (-1) }} & \H^{i-m+1}(\PP^p, \I^{j-m+1}, \LC(-1))  \ar[rr]^-{c(\Oo(1))^{m-1}} && \H^{i}(\PP^p,\I^{j}, \LC(-m) )
  }   
\]
where \(\partial_{\LC(-1)}\) is the connecting homomorphism (or Bockstein homomorphism)  
\cite[\S\,2.1]{fasel:ij} and \(c(\Oo(1))\) is the Euler class homomorphism \cite[\S\,3]{fasel:ij}. 

\subsection{$\I$-cohomology of completely split quadrics}\label{sec:I-additive-result}
The aim of this section is to obtain corresponding results for split quadrics, i.\,e.\ to determine the $\I$-cohomology of split quadrics in all bidegrees, with all twists.  Over a field, we will later also compute the ring structure; see Section~\ref{multiplicativecomp}.

We will always assume \(n \geq 3\). Recall that \(Q_1 \cong \PP^1\) (for which the previous computation applies) and that \(Q_2 \cong \PP^ 1 \times \PP^1\) (using the Segre embedding in \(\PP^ 3\)).

\begin{theorem}\label{thm:I-additive}
  For the split quadric \(Q_n\) of dimension \(n\geq 3\), we have isomorphisms of groups as follows:
  \[
    \Small \H^i( Q_n, \I^j) 
    \cong 
    \left\{
      \begin{array}{ll}
	\big(\bigoplus\limits_{m \in T_n} \H^{i-m}(S, \Ib^{j-m}) \big) \oplus \H^{i}(S,\I^j) \oplus \H^{i-n}(S,\I^{j-n})  & \textnormal{ if \(n = 2p\) and \(p\) is even} \\
	\big(\bigoplus\limits_{m \in T_n} \H^{i-m}(S, \Ib^{j-m}) \big) \oplus \H^{i}(S,\I^j)  \oplus \H^{i-p}(S,\I^{j-p})^{\oplus 2}  \oplus \H^{i-n}(S,\I^{j-n}) & \textnormal{ if \(n = 2p\) and \(p\) is odd}  \\
	\big(\bigoplus\limits_{m \in T_n} \H^{i-m}(S, \Ib^{j-m}) \big) \oplus \H^{i}(S,\I^j) \oplus \H^{i-p-1}(S,\I^{j-p-1})  & \textnormal{ if \(n = 2p+1\) and \(p\) is even}  \\
	\big(\bigoplus\limits_{m \in T_n} \H^{i-m}(S, \Ib^{j-m}) \big) \oplus \H^{i}(S,\I^j)  \oplus \H^{i-p}(S,\I^{j-p})  & \textnormal{ if \(n = 2p+1\) and \(p\) is odd}  \\
      \end{array} \right. 
  \]
  where \(T_n := \big\{ 1 \leq m \leq n :  m \textnormal{ even if $1 \leq m \leq \lfloor \frac{n}{2}\rfloor$ and }  m \textnormal{ odd if $\lceil \frac{n}{2} \rceil+1 \leq m \leq n$} \big\}\).
  \[
    \Small \H^i( Q_n, \I^j, \Oo(1)) 
    \cong 
    \left\{
      \begin{array}{ll}
	\big(\bigoplus\limits_{m \in U_n} \H^{i-m}(S, \Ib^{j-m}) \big) \oplus \H^{i-p}(S,\I^{j-p})^{\oplus 2} & \textnormal{ if \(n = 2p\) and \(p\) is even} \\
	\big(\bigoplus\limits_{m \in U_n} \H^{i-m}(S, \Ib^{j-m}) \big)  & \textnormal{ if \(n = 2p\) and \(p\) is odd}  \\
	\big(\bigoplus\limits_{m \in U_n} \H^{i-m}(S, \Ib^{j-m}) \big)  \oplus \H^{i-p}(S,\I^{j-p}) \oplus \H^{i-n}(S,\I^{j-n})  & \textnormal{ if \(n = 2p+1\) and \(p\) is even}  \\
	\big(\bigoplus\limits_{m \in U_n} \H^{i-m}(S, \Ib^{j-m}) \big) \oplus \H^{i-p-1}(S,\I^{j-p-1}) \oplus \H^{i-n}(S,\I^{j-n})  & \textnormal{ if \(n = 2p+1\) and \(p\) is odd}  \\
      \end{array} \right. 
  \]
  where \(U_n := \big\{ 1 \leq m \leq n :  m \textnormal{ odd if $1 \leq m \leq \lfloor \frac{n}{2}\rfloor$ and }  m \textnormal{ even if $ \lceil \frac{n}{2} \rceil+1 \leq m \leq n$} \big\}\).
\end{theorem}

The remainder of this section constitutes a proof of this theorem.  
The proof relies on homotopy invariance (Theorem~\ref{thm:homotopy-invariance}),
localization and d\'{e}vissage \cite[Th\'eor\`eme 9.3.4 and Remarque~9.3.5]{faselthesis}
for \(\I\)-cohomology. We need to distinguish cases based on the parities of \(n\), \(p\) and the twist, so there are eight different cases to consider. 

\subsection{Twisting homomorphisms for quadrics}\label{sec:twist}

\begin{definition}\label{def:twisthomo}
  Let \(\lambda^m_\LC \colon  \H^{i-m}(S, \Ib^{j-m})  \rightarrow \H^{i}(Q_n,\I^j, \Oo(-m)) \) denote the following composition:
  \[  
    \xymatrix{
      \H^{i-m}(S, \Ib^{j-m})  \ar[r]^-{q^*} & \H^{i-m}(Q_n, \Ib^{j-m}) \ar[r]^-{\partial_{\LC (-1) }} & \H^{i-m+1}(Q_n, \I^{j-m+1}, \LC(-1))  \ar[rr]^-{c(\Oo(1))^{m-1}} &  & \H^{i}(Q_n,\I^{j}, \LC(-m) )
    }   
  \]
  where \(\partial_{\LC(-1)}\) is the connecting homomorphism (or Bockstein homomorphism)  
  \cite[\S\,2.1]{fasel:ij} and \(c(\Oo(1))\) is the Euler class homomorphism \cite[\S\,3]{fasel:ij}. 
\end{definition}
This homomorphism is analogous to the map \(\mu^m_\LC\) defined by Fasel on projective spaces \eqref{eq:defn:mu_m}. 
\begin{lemma}\label{lemma:twisthomo}
  The following diagram commutes. 
  \[ 
    \xymatrix{
      \H^{i}(Q_n,\I^j, \LC(-m)) \ar[r]^-{j^*} 
      & \H^{i}(Q_n- \PPpX,\I^j, \LC(-m))     \\
      \H^{i-m}(S, \Ib^{j-m})  \ar[u]^-{\lambda^m_\LC}  \ar[r]^-{\mu^m_\LC}
      &\H^i(\PPpY, \I^j,  \LC(-m)) \ar[u]_{\rho^*}  }
  \]
\end{lemma}
\begin{proof}
  Similarly to \(\lambda^m_\LC\) for \(Q_n\) and \(\mu^m_\LC\) for \(\PP^p\), we can define twist homomorphisms \(\delta^m_\LC\colon \H^{i-m}(S, \Ib^{j-m}) \rightarrow \H^i(Q_n-\PPpX, \I^j, \Oo(-m))\) for \(Q-\PPpX\).  The result then follows as the Bockstein homomorphisms and the Euler class homomorphisms commute with the pullback homomorphisms \(j^*\) and \(\rho^*\), cf.\ \cite[Proposition~2.1 and \S\,3]{fasel:ij}.
\end{proof}
\subsection{The key short split exact sequence} If \(n=2p\) (resp. \(n=2p+1\)), we defined \(q:=p\) (resp. \(q:= p+1\)).  By d\'{e}vissage, homotopy invariance and localization, we obtain a long exact sequence
\[ 
  \xymatrix{
    \cdots \rightarrow  \H^{i-q}(\PPpX,\I^{j-q},  \omega_{x} \otimes \Oo(l) ) \ar[r]^-{(\iota_x)_*} &\H^{i}(Q_n,\I^j, \Oo(l)) \ar[r]^-{\iota_y^*} 
    & \H^{i}( \PPpY,\I^j, \Oo(l))    \ar[r]^-\partial & \H^{i-q+1}(\PPpX,\I^{j-q},  \omega_{x} \otimes \Oo(l) )  \rightarrow \cdots 
  }
\]
with \(\omega_x \cong \Oo(1-q)\) (see Lemma~\ref{lem:normal-bundles-in-Q}) and \(q\) the codimension of \( \PPpX\) in \(Q_n\).
In this section, we prove the following result:
\begin{theorem}\label{thm:keyexactsequence}
  The sequence
  \[ 
    \xymatrix{
      0 \ar[r] & \H^{i-q}(\PPpX,\I^{j-q},  \omega_{x} \otimes \Oo(l) ) \ar[r]^-{(\iota_x)_*} &\H^{i}(Q_n,\I^j, \Oo(l)) \ar[r]^-{\iota_y^*} 
      & \H^{i}( \PPpY,\I^j, \Oo(l))    \ar[r] & 0 
    }
  \]
  is split exact. 
\end{theorem}
\begin{proof}
  \textbf{Case I.} We consider first the case \(n = \dim(Q_n) = 2p\) with \(p\) even and \(l\) even. We will show that the map \(\iota_y^*\colon\H^{i}(Q_n,\I^j) \rightarrow 
  \H^{i}(\PPpY,\I^j) \) is split surjective. By homotopy invariance, the vector bundle projection \(\rho\) of Lemma~\ref{lem:rho-is-vb} induces an isomorphism
  \(
    \rho^* \colon  \H^i(\PPpY, \I^j) \rightarrow \H^{i}(Q_n- \PPpX,\I^j)
  \).
  Consider the following commutative diagram:
  \[ 
    \xymatrix{
      \H^{i}(Q_n,\I^j) \ar[r]^{j^*}  \ar[dr]^-{\iota_y^*}
      & \H^{i}(Q_n- \PPpX,\I^j)     \\
      \big(\bigoplus\limits_{m \textnormal{ even}}^{ 2\leq m \leq p } \H^{i-m}(S, \Ib^{j-m}) \big) \oplus \H^{i}(S,\I^j) \ar[u]^\alpha  \ar[r]^-\beta
      &\H^i(\PPpY, \I^j) \ar[u]_{\rho^*}  }
  \]
  The commutativity of the upper right triangle follows from the commutative diagram in Figure~\ref{fig:geometry} and \(i_y^* = (\rho^*)^{-1}\).
  Here, \(\beta\) is the isomorphism  \( \sum_m \mu^m + p^*\) of \cite[Theorem~9.1]{fasel:ij}, cf.\  Definitions~5.1 and 5.9 and Corollary~5.8 of loc.\ cit. 
  The map \(\alpha\) is defined similarly as \(\beta\), namely as \(\alpha := \sum_m \lambda^m + q^*\).  The diagram commutes by Lemma~\ref{lemma:twisthomo}. Since \(\beta\) and \(\rho^*\) are both isomorphisms, we obtain a splitting \( \alpha \circ \beta^{-1}\) of \(\iota_y^*\). 
  
  The cases 
  \begin{itemize}
  \item \(n=2p+1\), \(p\) even, \(l\) even
  \item  \(n=2p\), \(p\) odd, \(l\) odd
  \item  \(n=2p+1\), \(p\) odd, \(l\) odd
  \end{itemize}
  are proved similarly.
  
  \textbf{Case II.} Next, we consider the case \(n = \dim(Q_n) = 2p\) with \(p\) odd and \(l\) even. 
  Let \(Bl\) be the blow-up of \(Q_n\) along \(\PP^{p}_{x}\) and let \(E\) be the exceptional fibre. Recall from Proposition~\ref{pblup} above that this setup satisfies Hypothesis~1.2 of \cite{BC}. We use the same notation as in Proposition~\ref{pblup}. As in the previous case, we want to show that \(j^*\colon\H^{i}(Q_n,\I^j) \rightarrow \H^{i}(Q_n- \PPpX,\I^j) \) is split surjective. We draw the following diagram:
  \[ \xymatrix{
      \H^i_{\PPpX}(Q_n,\I^j) \ar[r] &\H^{i}(Q_n,\I^j) \ar[r]^-{j^*} 
      & \H^{i}(Q_n- \PPpX,\I^j)    \ar[r]^-{\partial} & \H^{i+1}_{\PPpX}(Q_n,\I^j)  \\
      &  \H^i(Bl, \I^j, \omega_\pi) \ar[u]^-{\pi_*}  \ar[d]^{\cong}
      &\H^i(\PPpY, \I^j) \ar[u]^-{\rho^*}_{\cong}   \ar[ld]^-{\tilde{\rho}^*} \\
      &  \H^i(Bl, \I^j)   }
  \]
  Using \cite[Proposition~A.11.(iii)]{BC}, we find that \(\omega_\pi \) is isomorphic to \(\Oo(p-1)\).  As \(\I\)-cohomology is two-periodic in the twist, and as \(p-1\) is even, we have an isomorphism as indicated by the lower vertical arrow.  We can push forward along \(\pi\) by an analogue of \cite[Proposition~2.1(A)]{BC}. Then, arguing as in the proof of  \cite[Theorem~2.3]{BC}, we see that the middle diagram is commutative: as \(\lambda(\Oo) = 0\) and as \(p\) is odd, \(\lambda(\Oo) \equiv p - 1 \mod 2\), as required. 
  
  The case \(n=2p\), \(p\) even and \(l\) odd is obtained similarly. In this case, we need to use Proposition~\ref{prop:lambda} to compute \(\lambda(\Oo(1)) = -1\).

  \textbf{Case III.} 
  Now, we consider the case \(n = 2p+1 = \dim(Q_n)\) is odd and \(p\) is odd and \(l\) is even.
  We show that \(\iota_{*}\colon \H^{i-p-1}(\PPpX, \I^{j-p-1}, \Oo(-p)) \rightarrow \H^{i}(Q_n,\I^j) \) is split injective. Consider the following diagram:
  \[ 
    \xymatrix{
      \H^{i-p-1}(\PPpX, \I^{j-p-1}, \Oo(-p)) \ar[r]^-{\iota_*} \ar[d]^-{\red_{\Oo(-p)}} & \H^{i}(Q_n,\I^j)  \ar[d]^{\red} \\
      \H^{i-p-1}(\PPpX, \Ib^{j-p-1})  \ar[r]^-{\iota_*} & \H^{i}(Q_n,\Ib^j)
    }
  \]
  The diagram is commutative since pushforward commutes with mod-\(2\)-reduction. The lower horizontal map \(\iota_*\colon \H^{i-p-1}(\PPpX, \Ib^{j-p-1})  \rightarrow \H^{i}(Q_n,\Ib^j)\) is split injective because \(\Ib\)-cohomology is oriented and because \(\iota_*\) is part of a localization sequence.  Let \(s_1\colon  \H^{i}(Q_n,\Ib^j) \rightarrow \H^{i-p-1}(\PPpX, \Ib^{j-p-1})\) denote a splitting.  The map \(\red_{\Oo(-p)}\) is also split injective: a splitting is given by the commutative diagram
  \[ 
    \xymatrix{
      \H^{i-p-1}(\PPpX, \Ib^{j-p-1})  & \ar[l]_-{\red_{\Oo(-p)}} \H^{i-p-1}(\PPpX, \I^{j-p-1}, \Oo(-p))  \\
      \bigoplus\limits_{0\leq m \leq p} \H^{i-p-1-m}(S, \Ib^{j-p-1-m}) \ar[r]^-{\textnormal{pr}} \ar[u]^-{\cong}_-{\sum \bar{\mu}^m} & \bigoplus\limits^{0\leq m \leq p}_{m \textnormal{ odd} } \H^{i-p-1-m}(S, \Ib^{j-p-1-m}) \ar[u]_-{\cong}^-{\sum \mu^m}  
    }
  \] 
  The commutativity and vertical isomorphisms are all established in \cite{fasel:ij}. Let this splitting be denoted by \(s_2\). Now, the composition 
  \[ 
    \H^i(Q_n, \I^j) \xrightarrow{\red} \H^{i}(Q_n,\Ib^j) \xrightarrow{s_1}  \H^{i-p-1}(\PPpX, \Ib^{j-p-1}) \xrightarrow{s_2} \H^{i-p-1}(\PPpX, \I^{j-p-1}, \Oo(-p))  
  \]
  provides a splitting for \(\iota_*\colon  \H^{i-p-1}(\PPpX, \I^{j-p-1}, \Oo(-p)) \rightarrow  \H^{i}(Q_n,\I^j) \).
  
  The case \(n=2p+1\), \(p\) even, \(l\) odd is obtained similarly.
\end{proof}

\begin{proof}[Proof of Theorem~\ref{thm:I-additive}]
  By Theorem~\ref{thm:keyexactsequence}, we can compute the \(\I\)-cohomology of quadrics from the \(\I\)-cohomology of projective spaces (Theorem~\ref{thm:Faselproj}). 
\end{proof}

\section{$\I$-cohomology: multiplicative structure}
\label{multiplicativecomp}

In this section, our base \(S\) is always the spectrum of a field \(F\) of characteristic \(\neq 2\).  Given a smooth scheme \(X\) over \(F\), we define\footnote{A reader worried about the indexing set \(\Pic(X)/2\) may
  consult \cite{BC12}.}
\begin{align*}
  \H^\tot(X, \I) &:=  \bigoplus_{\mathclap{\substack{i,j \in \ZZ \\ \LC  \in \Pic(X)/2}}}  \H^i(X, \I^j, \LC )\\
  \H^{\star,\bullet}(X, \Ib) &:=  \bigoplus_{\mathclap{i,j \in \ZZ}}  \H^i(X, \Ib^j)
\end{align*}
The \(\Ib\)-cohomology ring is a commutative \(\ZZ\oplus\ZZ\)-graded ring. The \(\I\)-cohomology ring is \(\ZZ\oplus\ZZ\oplus(\Pic(X)/2)\)-graded, and it is graded commutative in the sense that 
\[
  aa' = (-1)^{ii'}a'a 
\]
for homogeneous elements of degrees \(\deg{a} = (i,j,\bar l)\) and \(\deg{a'} = (i',j',\bar l')\), respectively.  For \(X = \PP^n\) or \(X = Q_n\), we have \(\Pic(X) \cong \ZZ\), so that the \(\I\)-cohomology ring is a \(\ZZ \oplus \ZZ \oplus \ZZII\)-graded algebra over the \(\ZZ \oplus \ZZ \oplus \ZZII\)-graded coefficient ring
\[
  \mathbb{I}:= \bigoplus_{\mathclap{\substack{i,j \in \ZZ\\\bar l \in \ZZII}}} \H^i(F, \I^j,\bar l)
\] 
concentrated in degrees \((0,*,\bar 0)\):
\[
  \H^i(F, \I^j, l) := \left\{	
    \begin{array}{ll}
      \I^j(F) & \textnormal{if \(i=0\) and \(l\equiv 0 \mod 2\)}\\
      0 & \textnormal{otherwise}
    \end{array} 
  \right.
\]
Note that we have a graded ideal \(\mathbb{I}':=(\H^0(F,\I^1,0) ) \cdot \mathbb{I} \subset \mathbb{I}\).

\subsection{The ring $\H^\tot(\PP^p, \I)$}
\begin{theorem}\label{alternativeproof}
  Consider \( \PP^p = \PP^p_F\), where \(F\) is a field of characteristic \(\neq 2\). For any \(p \geq 1\), we have a \(\ZZ \oplus \ZZ \oplus \ZZII\)-graded ring isomorphism
  \begin{align*}
    \H^\tot(\PP^p, \I)
    &\cong \mathbb{I}[\xi,\alpha]/(\mathbb{I}'\xi, \xi^{p+1},\xi\alpha,\alpha^2)
      \quad\text{ with }
      \begin{cases}
	\deg{\xi} = \twistedtrideg{1}{1}{1}\\
	\deg{\alpha} = \twistedtrideg{p}{p}{p+1}
      \end{cases}
  \end{align*}
\end{theorem}
\begin{proof}
  From the additive result of Fasel quoted as Theorem~\ref{thm:Faselproj} above, we conclude:
  \begin{equation}\label{eq:I-cohomology-of-P-over-field}
    \H^i(\PP^p, \I^j, \Oo(l)) \cong \left\{  
      \begin{array}{lll}
	\FIb^{j-i}(F) & \textnormal{if \( 1\leq i \leq p\) and \(l \equiv i \mod 2\) } \\
	\FI^{j-i}(F) & \textnormal{if \( i=0 \) and \(l\) is even} \\
	\FI^{j-i}(F) & \textnormal{if \( i=p \) and \(l \equiv -p-1 \mod 2\)} \\
	0 & \textnormal{otherwise}
      \end{array} 
    \right.
  \end{equation}
  Let \(\xi\) be the generator of \( \H^1(\PP^p, \I^1,\Oo(-1)) \cong \ZZII\), and let more generally \(\xi_m \in \H^m(\PP^p,\I^m,\Oo(-m))\) be the additive generator for \(1\leq m \leq p\).  This generator is the image of \(1\) under the map \(\mu^m\) (cf.\  \eqref{eq:defn:mu_m}).  
  
  We first check that the product \(\phi\xi_m = p^*\phi\cup \xi_m\) vanishes for an element \(\phi\in \I^j(F)\) if and only if \(\bar\phi=0 \in \Ib^{j}(F)\), i.e.\ if and only if \(\phi\in\I^{j+1}(F)\). To see this, consider the following diagram, in which the horizontal maps in the right square are the mod-\(2\)-reductions:
  \[
    \xymatrix{ \H^0(F, \I^j) \times \H^0(F, \Ib^0) \ar[d]^-{\cup} \ar[r]^-{(p^*, \mu^m)} &  \ar[d]^-{\cup}  \H^0(\PP^p, \I^{j}) \times \H^m(\PP^p, \I^m, \Oo(-m)) \ar[r]^-{\red \times \red} & \H^0(\PP^p, \Ib^j) \times \H^m(\PP^p, \Ib^m) \ar[d]^-{\cup} \\
      \H^0(F, \Ib^j)   \ar[r]^-{\mu^m} &  \H^{m}(\PP^p, \I^{j+m}, \Oo(-m)) \ar[r]^{\red} & \H^{m}(\PP^p, \Ib^{j+m}) } 
  \]
  The whole diagram commutes. Indeed, Lemma~5.2 of \cite{fasel:ij} identifies the compositions \(\red\circ\mu^m\) with multiplication with \(\xi^m\).  (Note that we are applying the lemma in the case when \(i=0\) and Fasel's base scheme \(X\) is \(\Spec(F)\), the spectrum of a field.  The second summand on the right side of the formula given in the lemma is therefore zero since \(\partial_{\mathcal L}(\alpha)\) lies in the trivial group \(\H^1(F,\I^{j+1})\).)  The commutativity of the outer square of the diagram therefore follows from the ring structure on \(\H^{\star,\bullet}(\PP^p, \Ib)\).   The commutativity of the square in the right half of the diagram is clear.  Moreover, we know from the additive computations that both factors in the lower horizontal composition are isomorphisms. In particular, the commutativity of the square on the left follows from the commutativity of the other two squares.

  Now suppose \(p^*\phi\cup \xi_m = 0\).  As \(\xi_m= \mu^m(1)\), the commutativity of the left square in the above diagram shows that this is equivalent to the condition \(\mu^m(\bar\phi)=0\).  As \(\mu^m\) is an isomorphism, this in turn is equivalent to \(\bar\phi = 0 \in \Ib^j(F)\), as claimed.
  
  A similar argument as above also shows the commutativity of the following diagram:
  \[ 
    \xymatrix{ 
      \H^0(F, \Ib^0) \times \H^0(F, \Ib^0) \ar[d]^-{\cup} \ar[r]^-{(\mu^l, \mu^m)} & \ar[d]^-{\cup}  \H^l(\PP^p, \I^l, \Oo(-l)) \times \H^m(\PP^p, \I^m, \Oo(-m)) \\
      \H^0(F, \Ib^0)   \ar[r]^-{\mu^{l+m}}                                         & \H^{l+m}(\PP^p, \I^{l+m}, \Oo(-l-m))   
    } 
  \]
  This shows that \(\xi_l\xi_m = \xi_{l+m}\) in the appropriate degrees, and hence that \(\xi_m = \xi_1^m\) is an additive generator for \(1\leq m\leq p\).  Note that \(\xi^{p+1} =0\) for degree reasons. 
  
  Finally, let \(s\colon S \rightarrow \PP^p\) be a rational point, and let \(\alpha\) be the image of \(1\in \W(F) = \H^0(F,\I^0)\) under the isomorphism \(s_*\colon \H^0(F,\I^0)\to \H^p(\PP^p, \I^p, \Oo(-p-1))\).   Again for degree reasons, \(\xi\alpha = 0\) and \(\alpha^2 = 0\).  It remains to check that \(\alpha\phi = \alpha\cup p^*\phi\) vanishes for an element \(\phi\in \H^i(F,\I^j,k)\) if and only if \(\phi = 0\).  This is immediate from the identity \(\alpha \cup p^*\phi = s_* (\phi)\), which in turn follows from the projection formula of \cite{calmesfasel}:
  for any \(\phi\in \H^0(F,\I^j)\), this formula gives
  \(
  s_*\psi \cup p^*\phi 
  = s_* (\psi \cup s^* p^* \phi)
  \),
  and we conclude by taking \(\psi = 1\) and noting that \( p \circ s = \id\).
\end{proof}

\subsection{Milnor cohomology and \(\Ib\)-cohomology}
\begin{theorem}[Karpenko-Merkurjev, Dugger-Isaksen]\label{CHquadric}
  The Chow rings of the split quadrics \(Q_n\) over a field \(F\) can be described as follows:
  \[
    \CH^\bullet(Q_n) \cong 
    \begin{cases}
      \ZZ[x,y]/(x^{p+1} -2xy,y^2 -x^{p}y) & \textnormal{ if \(n=2p\) with \(p\) even } \\
      \ZZ[x, y]/(x^{p+1} - 2xy, y^2) & \textnormal{ if \(n=2p\) with \(p\) odd } \\
      \ZZ[x,y]/(x^{p+1} - 2y, y^2) & \textnormal{ if \(n=2p+1\) }
    \end{cases}
  \]
  where \(\deg{x}=1\) and \(\deg{y}=p\) if \(n\) is even, \(\deg{y}=p+1\) if \(n\) is odd.
\end{theorem}
\begin{proof}
  See \cite[Theorems~A.4 and A.10]{DI:Hopf}.
\end{proof}

\begin{corollary}\label{coro:MilnorQ} 
  The bigraded Milnor sheaf cohomology rings \(\H^{\star,\bullet}(Q_n, \KM) :=  \bigoplus_{i,j \in \ZZ }  \H^i(Q_n, \KM_j)\) of the split quadrics \(Q_n\) over a field \(F\) can be described as follows:
  \[
    \H^{\star,\bullet}(Q_n, \KM) \cong 
    \begin{cases}
      \H^{\star,\bullet}(F,\KM)[x,y]/(x^{p+1} -2xy,y^2 -x^{p}y) & \textnormal{ if \(n=2p\) with \(p\) even } \\
      \H^{\star,\bullet}(F,\KM)[x, y]/(x^{p+1} - 2xy, y^2) & \textnormal{ if \(n=2p\) with \(p\) odd } \\
      \H^{\star,\bullet}(F,\KM)[x,y]/(x^{p+1} - 2y, y^2) & \textnormal{ if \(n=2p+1\) }
    \end{cases}
  \]
  where \(\deg{x}=(1,1)\) and \(\deg{y}=(p,p)\) if \(n\) is even, \(\deg{y}=(p+1,p+1)\) if \(n\) is odd.
\end{corollary}
\begin{proof}
  In the computation of the bigraded motivic cohomology ring of quadrics \cite[Proposition~4.3]{DI:Hopf} we can replace motivic cohomology with Milnor sheaf cohomology. To pass from the geometric bidegrees to the full bigraded ring, note that \(\H^{\star,\bullet}(Q_n, \KM)\) is a free module over \(\H^{\star,\bullet}(F, \KM)\) with generators in geometric bidegrees \((0,0)\), \((1,1)\), \ldots, \((n,n)\) when \(n\) is odd, and in geometric bidegrees \((0,0)\), \((1,1)\), \ldots, \((n,n)\) plus an extra generator in bidegree \((\frac{n}{2},\frac{n}{2})\) when \(n\) is even. This is proved by an analogous argument as in \cite[Proposition~4.1]{DI:Hopf}, using localization and homotopy invariance of bigraded Milnor cohomology.
\end{proof}

\begin{corollary}\label{thm:I-ring-without-twists}
  We have isomorphisms of graded commutative rings as follows:
  \[ 
    \H^{\star,\bullet}(Q_n, \Ib) \cong \left\{ 
      \begin{array}{llll}
        \H^{\star,\bullet}(F,\Ib)[\bar{\xi},\bar{\beta}]/(\bar{\xi}^{p+1}, \bar{\beta}^2 - \bar{\xi}^{p}\bar{\beta})     & \textnormal{if \(n=2p\) with \(p\) even}             \\
        \H^{\star,\bullet}(F,\Ib)[\bar{\xi},\bar{\beta}]/(\bar{\xi}^{p+1},\bar{\beta}^2)          & \textnormal{if \(n=2p+1\) or (\(n=2p\) with \(p\) odd)}
      \end{array}  
    \right. 
  \]
  Here, \(\deg{\bar{\xi}} = (1,1)\) and \(\deg{\bar{\beta}} = (q,q)\) where \(q = p\) if \(n=2p\) is even and \(q = p+1\) if \(n=2p+1\) is odd.
\end{corollary}
\begin{proof}
  By the Milnor conjecture, we can identify \(K^M_\bullet(F)/2\) with \(\Ib^\bullet(F)\) via the Pfister norm map, and we can further identify the ring \(\H^{\star,\bullet}(Q_n, \Ib)\) with \(\H^{\star,\bullet}(Q_n, \KM/2)\), cf.\ \cite[Proposition~4.11]{fasel:chowwittring}.
  The result then follows from Corollary~\ref{coro:MilnorQ}.
\end{proof}
\begin{remark}\label{rem:generators-of-DI}
  If \(n\) is odd, the generators \(\bar{\xi}\) and \(\bar{\beta}\) are uniquely defined by the explicit computation of \(\Ib\)-cohomology above, so there is no ambiguity in choosing the generators. If \(n\) is even, the generator \(\bar{\xi}\) is still uniquely defined, but  the generator \(\bar{\beta}\) has an ambiguity because \(\H^p(Q_{2p},\Ib^p ) \cong \ZZII \oplus \ZZII\). Consider the following short split exact sequence:
  \begin{equation*}
    \xymatrix{
      0 \ar[r] & \H^{0}(\PPpX,\Ib^{0}) \ar[r]^-{(\iota_x)_*} &\H^{p}(Q_{2p},\Ib^p) \ar[r]^-{(\iota_y)^*} 
      & \H^{p}(\PPpY,\Ib^p)    \ar[r] & 0 \\
      & & \H^0(\PP_{x'}^p,\Ib^{0}) \ar[u]^-{(\iota_{x'})_*} \ar[r]^-{(s_{x'})^*}_-{\cong}  & \H^{0}(F,\Ib^{0}) \ar[u]^-{(s_y)_*}_-{\cong}
    }
  \end{equation*}
  The subschemes \(\PP^p_y\) and \(\PP^p_{x'}\) of \(Q_{2p}\) intersect transversally in a point, so we have a cartesian square as in diagram \eqref{n2p:cartesian} below. The induced pullback and pushforward morphisms fit into a square as in the diagram above, and the base change formula shows that this square commutes.  We choose \(\bar{\beta}\) and \(\bar{\alpha}\) to be \((\iota_x)_*(1)\) and \((\iota_{x'})_*(1)\), respectively.  Note that the pushforward maps of \(\Ib\)-cohomology and \(\KM/2\)-cohomology are compatible via the norm map \cite[Lemma~10.4.4]{faselthesis}.  It follows that our generators \(\bar{\beta}\) and \(\bar{\alpha}\) map to the mod-\(2\)-reductions of Dugger-Isaksen's generators \(\alpha\) and \(\beta\) in \cite[Section~A.5]{DI:Hopf}, respectively. Therefore, we have also the relations \(\bar{\alpha}+\bar{\beta} = \bar{\xi}^p \), \(\bar{\beta}^2 = \bar{\xi}^p\bar{\beta}\) if \(p\) is even, and \(\bar{\beta}^2 = 0\) if \(p\) is odd in \(\Ib\)-cohomology, as in the proof of \cite[Theorem~A.10]{DI:Hopf}. 
\end{remark}
\subsection{The ring $\H^\tot(Q_n, \I)$}
In the following computations, we use the base change theorem for \(\I\)-cohomology (Theorem~\ref{thm:base-change-formula}) and the projection formula \cite[\S\,2.2]{fasel:ij} several times without explicitly mentioning them.  This requires even more care than in the additive case when dealing with twists by line bundles, so we will be even more explicit about these.

Let \(q = p\) if \(n=2p\) is even and \(q = p+1\) if \(n=2p+1\) is odd. Recall the split exact sequence
\begin{equation}\label{seq:keysequence-Q_2p}
  \begin{gathered} 
    \xymatrix{
      0 \ar[r] & \H^{i-q}(\PPpX,\I^{j-q},  \omega_x \otimes \Oo(l) ) \ar[r]^-{(\iota_x)_*} &\H^{i}(Q_n,\I^j, \Oo(l)) \ar[r]^-{\iota_y^*} 
      & \H^{i}( \PPpY,\I^j, \Oo(l))    \ar[r] & 0 
    }
  \end{gathered}
\end{equation}
with \(\omega_x \cong \Oo(1-q)\) of Theorem~\ref{thm:keyexactsequence}.  Note that, strictly speaking, we have identified $\Oo(l)$ with $(\iota_y)^*\Oo(l)$ in this exact sequence.  This should not cause any confusion: for any closed immersion $\iota\colon X \rightarrow Y$ of smooth closed subschemes of a fixed projective space, we have a \emph{canonical} isomorphism $\iota^*\Oo_Y(l) \cong \Oo_X(l)$.  We will make this simplification of notation throughout this section whenever possible, without further comment.

\begin{theorem}\label{thm:I-ring}
  Let \(F\) be a field of characteristic \(\neq 2\).
  For \(p \geq 1 \), we have a \(\ZZ \oplus \ZZ \oplus \ZZII\)-graded ring 
  (and even \(\II\)-algebra) isomorphism
  \[\H^\tot(Q_{2p+1}, \I)  \cong 
    \II[\xi,\alpha,\beta]/(\II'\xi,\xi^{p+1},\xi\alpha,\alpha^2,\beta^2)
  \]
  Here, the generators are of degrees
  \(\deg{\xi}    = \twistedtrideg{1}{1}{1}\),
  \(\deg{\alpha} = \twistedtrideg{p}{p}{p-1}\), and
  \(\deg{\beta}  = \twistedtrideg{p+1}{p+1}{p}\). The reduction map \(	\H^\tot(Q_{2p+1}, \I) \rightarrow \H^{\star,\bullet}(Q_{2p+1}, \Ib),\)
  which collapses the \(\ZZII\)-grading, is given by \(\xi \mapsto \bar{\xi}\), \( \alpha\mapsto  \bar{\xi}^p \) and \(\beta \mapsto \bar{\beta}\).
  
\end{theorem}

\begin{proof}
  From the split exact sequence \eqref{seq:keysequence-Q_2p} for \(n=2p+1\), or by specializing Theorem~\ref{thm:I-additive} to the case \(S=\Spec(F)\), we obtain:
  \[ 
    \H^i(Q_{2p+1}, \I^j, \Oo(l)) = \left\{  
      \begin{array}{lll}
	\FIb^{j-i}(F) & \textnormal{if \( 1\leq i \leq p\) and \(l\equiv i\mod 2\) } \\
	\FIb^{j-i}(F) & \textnormal{if \( p+2\leq i \leq 2p+1\) and \(l\not\equiv i\mod 2\)} \\
	\FI^{j-i}(F)  & \textnormal{if (\(i=p\) or \(i=p+1\)) and \(l \not\equiv i \mod 2\)} \\
	\FI^{j-i}(F) & \textnormal{if (\(i=0 \) or \(i=2p+1\)) and \(l\equiv i \mod 2\)} \\
	0 & \textnormal{otherwise}
      \end{array} 
    \right.
  \] 
  
  We claim that the reduction map \(\H^i(Q_n, \I^j, \Oo(l)) \rightarrow \H^i(Q_n, \Ib^j) \cong \FIb^{j-i}(F) \)
  is an isomorphism for \( 1\leq i \leq p\) and \(l \equiv i \mod 2\), and also for \( p+2\leq i \leq 2p+1\) and \(l \not \equiv i \mod 2\). 
  Indeed, in the first case, it follows from degree considerations and \eqref{eq:I-cohomology-of-P-over-field} that the map \(\iota_y^*\) in sequence~\eqref{seq:keysequence-Q_2p} is an isomorphism; in the second case, the map \((\iota_x)_*\) in this sequence is an isomorphism for degree reasons again. Consider the following commutative diagrams:
  \[
    \xymatrix{ \H^i(Q_n,\I^j, \Oo(l)) \ar[r] \ar[d]^-{\iota_x^*} & \H^i(Q_n,\Ib^j) \ar[d]^-{\iota_x^*}     &&     \H^{i-p-1}(\PPpY,\I^{j-p-1}, \omega_y \otimes \Oo(l))  \ar[d]^-{(\iota_y)_*} \ar[r] & \H^{i-p-1}(\PPpY,\Ib^{j-p-1}) \ar[d]^-{(\iota_y)_*}    \\
      \H^i(\PPpX,\I^j, \Oo(l)) \ar[r] &  \H^i(\PPpX,\Ib^j)                                      &&    \H^i(Q_n,\I^j, \Oo(l)) \ar[r]  & \H^i(Q_n,\Ib^j)   
    } 
  \]
  Using Fasel's additive computations for projective spaces \eqref{eq:I-cohomology-of-P-over-field} once again, we see that the lower arrow of the left square is an isomorphism for \( 1\leq i \leq p\) and \(l\equiv i \mod 2\), and that the upper arrow of the right square is an isomorphism for \( p+2\leq i \leq 2p+1\) and \(l\not\equiv i \mod 2\). The claim follows.
  
  Let $\xi\in \H^1(Q_n, \I^1, \Oo(1)) \cong  \ZZII $ be the unique generator, and choose generators $\alpha$ and $\beta$ of  $\H^p(Q_n, \I^p, \Oo(p-1)) \cong \W(F)$ and $\H^{p+1}(Q_n, \I^{p+1}, \Oo(p)) \cong \W(F)$, respectively. We see that the reduction map sends \(\xi \) to \( \bar{\xi}\), \( \alpha\) to \(  \bar{\xi}^p \) and \(\beta \) to \( \bar{\beta}\). The relations \(\II'\xi\),  \(\xi^{p+1}\) and \(\xi\alpha\) in the ring structure follow from the additive results and by considering reduction to \(\H^{\star,\bullet}(Q_n, \Ib)\). The relation \( \beta^2 = 0 \) is clear for degree reasons (\(\H^{2p+2}(Q_{2p+1}, \I^{2p+2}, \Oo(l)) =0\) since \(2p+2 > \dim Q\)), and similarly \(\alpha^2 = 0\) since \(\H^{2p}(Q_{2p+1}, \I^{2p}, \Oo(2p-2)) = 0\).
\end{proof}

We now consider the even-dimensional split quadrics over \(F\), i.e.\ \(Q_n\) with \(n=2p\).  In this case \(p=q\) in sequence~\eqref{seq:keysequence-Q_2p}. We choose the following additive generators:
\begin{itemize}
\item
  Let \(\xi \in \H^1(Q_{2p}, \I^1, \Oo(-1)) \cong  \FIb^0(F) \cong \ZZII \) be the unique generator.
\item
  Let \(\beta \in \H^{p}(Q_{2p}, \I^{p}, \Oo(p-1)) \) be the image of \(1 = \langle 1\rangle \in \W(F)\) under the following morphism of \(\W(F)\)-modules:
  \[   
    \W(F) \xrightarrow{p^*} \H^0(\PPpX,\I^0) \xrightarrow{t_x} \H^{0}(\PPpX,\I^{0}, \omega_x \otimes \Oo(p-1))  \xrightarrow{(\iota_x)_*}
    \H^{p}(Q_{2p}, \I^{p}, \Oo(p-1))
  \]
  The second map is defined by an isomorphism  \( \Oo \rightarrow \omega_x\otimes \Oo(p-1)\). We will write \(1_x\) for the element \(1_x := t_x p^*1 \in \H^{0}(\PPpX,\I^{0}, \omega_x \otimes \Oo(p-1))\), and similarly for the other embeddings \(\iota_{x'}\), \(\iota_y\), etc.
\item
  Let \(\alpha \in \H^{p}(Q_{2p}, \I^{p}, \Oo(p-1))\) similarly be defined as
  \(\alpha := (\iota_{x'})_* (1_{x'}) = (\iota_{x'})_*  t_{x'}  p_{x'}^*(1)\):
  \[   
    \W(F) \xrightarrow{p^*} \H^0(\PP^p_{x'},\I^0) \xrightarrow{t_{x'}} \H^{0}(\PP^p_{x'},\I^{0}, \omega_{x'} \otimes \Oo(p-1))  \xrightarrow{(\iota_{x'})_*}
    \H^{p}(Q_{2p}, \I^{p}, \Oo(p-1)), 
  \]
  where \(t_{x'} \) is defined by an isomorphism \( \Oo \rightarrow \omega_{x'}\otimes \Oo(p-1)\). 
\end{itemize}
We will choose the isomorphisms defining \(t_x\) and \(t_x'\) more carefully in the proof of the following theorem.  However, these choices have no effect on the final result.

\begin{theorem}\label{thm:I-ring-2}
  Let \(F\) be a field of characteristic \(\neq 2\). For any \(p\geq 2\), we have a \(\ZZ \oplus \ZZ \oplus \ZZII\)-graded ring 
  (and even \(\II\)-algebra) isomorphism
  \[
    \H^\tot(Q_{2p}, \I)  \cong  
    \begin{cases}
      \II[\xi,\alpha,\beta]/(\II'\xi,\xi^{p+1},\xi\alpha+\xi\beta,\alpha^2 + \beta^2,\alpha\beta) & \text{ if \(p\) is even} \\
      \II[\xi,\alpha,\beta]/(\II'\xi,\xi^{p+1},\xi\alpha+\xi\beta,\alpha^2,\beta^2) & \text{ if \(p\) is odd} 
    \end{cases}
  \]
  As can be seen from the explicit definitions above, the generators here have degrees
  \(\deg{\xi}    = \twistedtrideg{1}{1}{1}\) and 
  \(\deg{\alpha} = \deg{\beta}    = \twistedtrideg{p}{p}{p-1}\).
  The reduction map 
  \(\H^\tot(Q_{2p}, \I) \rightarrow \H^{\star,\bullet}(Q_{2p}, \Ib),\) 
  which collapses the \(\ZZII\)-grading, is given by
  \(\xi \mapsto \bar{\xi}\),
  \(\beta \mapsto \bar{\beta}\)
  and 
  \(\alpha \mapsto  \bar{\xi}^p - \bar{\beta}\).
\end{theorem}
\begin{proof}
  From the split exact sequence \eqref{seq:keysequence-Q_2p} or by specializing Theorem~\ref{thm:I-additive},  we obtain:
  \[ 
    \H^i(Q_{2p}, \I^j, \Oo(l)) = \left\{  
      \begin{array}{lll}
	\FIb^{j-i}(F) & \textnormal{if \( 1\leq i \leq p\) and \(i\equiv l \mod 2\) } \\
	\FIb^{j-i}(F)  & \textnormal{if \( p+1\leq i \leq 2p\) and \(i \not\equiv l \mod 2\)} \\
	\FI^{j-i}(F) \oplus \FI^{j-i}(F) & \textnormal{if \(i=p\) and \(i \not\equiv l \mod 2\)} \\
	\FI^{j-i}(F) & \textnormal{if (\(i=0\) or \(i=2p\)) and \(l\equiv i\)} \\
	0 & \textnormal{otherwise}
      \end{array} 
    \right.
  \]
  
  Analogously to the case when \(n = 2p+1\) is odd, we note that the reduction map \(\H^i(Q_{2p}, \I^j, \Oo(l)) \rightarrow \H^i(Q_{2p}, \Ib^j) \cong \FIb^{j-i}(F) \)
  is an isomorphism for \( 1\leq i \leq p\) and \(l \equiv i\), and also for \( p+1\leq i \leq 2p\) and \(l \not \equiv i\). The reduction map sends \(\xi\) to \(\bar{\xi}\), \(\alpha\) to \(\bar{\xi}^p - \bar{\beta}\) and \(\beta\) to \(\bar{\beta}\). 
  
  Let us now compute the ring structure of \(\H^\tot(Q_{2p},\I)\).  The relations \(\II'\xi,\xi^{p+1} \) and \(\xi(\alpha+\beta)\) follow by using the reduction map as before. 
  
  Let \(s_{x}\), \(s_y\),  \(s_{x'}\) and \(s_{y'}\) denote the inclusions of a point into \(\PP^p\) determined by the following cartesian diagrams:
  \begin{equation}\label{n2p:cartesian}
    \begin{aligned}
      \xymatrix{\Spec(F) \ar[d]^-{s_{x'}} \ar[r]^-{s_{y}} & \PPpY  \ar[d]^-{\iota_y} && \Spec(F) \ar[d]^-{s_{y'}} \ar[r]^-{s_{x}} & \PPpX  \ar[d]^-{\iota_{x}} \\
        \PP^p_{x'} \ar[r]^-{\iota_{x'}} & Q_{2p} && 	\PP^p_{y'} \ar[r]^-{\iota_{y'}} & Q_{2p}
      }  
    \end{aligned}
  \end{equation}
  Let \(s\) and \(\tilde{s}\) denote the rational points \(\Spec(F) \rightarrow Q_{2p}\) defined by the diagonal of the left and right diagram, respectively.	
  For \(i=p\) and \(l = 1-p\), we can use the maps in the left cartesian diagram to make the splitting in Theorem~\ref{thm:keyexactsequence} explicit:
  \begin{equation*}
    \xymatrix{
      0 \ar[r] & \H^{0}(\PPpX,\I^{j-p}) \ar[r]^-{(\iota_x)_*} &\H^{p}(Q_{2p},\I^j, \Oo(p-1)) \ar[r]^-{(\iota_y)^*} 
      & \H^{p}(\PPpY,\I^j, \Oo(p-1))    \ar[r] & 0 \\
      & & \H^0(\PP_{x'}^p,\I^{j-p}) \ar[u]^-{(\iota_{x'})_*} \ar[r]^-{(s_{x'})^*}_-{\cong}  & \H^{0}(F,\I^{j-p}) \ar[u]^-{(s_y)_*}_-{\cong}
    }
  \end{equation*}
  It follows that \(\alpha\) and \(\beta\) form a basis of the degree \(p\) part \( \H^{p}(Q_{2p}, \I^{p}, \Oo(p-1)) \) considered as a rank two free \(\W(F)\)-module. 
  In order to compute the various products of the generators \(\alpha\) and \(\beta\),
  let us consider a general element
  \[
    (a, b) \in \H^0(\PP^p_{x'}, \I^0, \omega_{x'} \otimes  \Oo(p-1)) \times \H^0(\PPpY, \I^0, \omega_{y} \otimes  \Oo(p-1)).
  \]
  Using the base change formula \ref{thm:base-change-formula} for the left cartesian diagram in  \eqref{n2p:cartesian} and the projection formula, we find:
  \begin{equation}\label{n=2p:cd1} 
    \begin{array}{llll}
      (\iota_{x'})_* a \cup (\iota_{y})_*  b 
      & = & (\iota_{x'})_* \big( a \cup \iota_{x'}^* (\iota_{y})_*b  \big)
      \\
      & = & (\iota_{x'})_* \big( a \cup  (s_{x'})_* s_y^* b   \big)
      \\
      & = &   (\iota_{x'})_*(s_{x'})_* \big( s_{x'}^* a \cup s_{y}^* b   \big)
      \\
      & = & s_*( s_{x'}^* a \cup s_{y}^* b)
    \end{array}
  \end{equation}
  Similarly, using the right diagram in \eqref{n2p:cartesian}, we find
  \[
    (\iota_{x})_* a \cup (\iota_{y'})_*  b  = \tilde{s}_*(s_{x}^* a \cup s_{y'}^* b).
  \] 
  
  \emph{When \(n=2p\) with odd \(p\)}, we consider the \(\AA^1\)-homotopy 
  \(
  h\colon \PP^p \times \mathbb{A}^1 \rightarrow Q_{2p} \times \mathbb{A}^1
  \)
  given by 
  \[ 
    h([a_0,a_1, \ldots, a_p],t) := ([ ta_1,  -t a_0, \ldots, ta_p, -t a_{p-1}; (1-t)a_0, (1-t)a_1, \ldots, (1-t)a_{p-1}, (1-t)a_p   ], t).
  \]
  This homotopy fits into the following commutative diagram consisting of two cartesian squares:
  \begin{equation}\label{diag:homotopy}
    \begin{aligned}
      \xymatrix{
        *+{\PP^p} \ar[r]^{h_0} \ar_{\id\times 0}[d] & Q_{2p}  \ar^{\id\times 0}[d] \\
        \PP^p\times \AA^1 \ar[r]^{h} & Q_{2p} \times \AA^1 \\
        *+{\PP^p}\ar[r]_{h_1} \ar^{\id\times 1}[u] & Q_{2p}  \ar_{\id\times 1}[u]
      }
    \end{aligned}
  \end{equation}
  Note that \(h_0 = \iota_y\) and \(h_1 = \iota_x \phi\) for some linear change of coordinates \(\phi\colon \PP^p\to\PP^p\). Using the base change formula and homotopy invariance, we find that the pushforward maps along \(h_0\) and \(h_1\) coincide up to an identification of the respective normal bundles.  Thus, up to an identification of the respective normal bundles, \((\iota_x)_*\) and \((\iota_y)_*\) also agree up to some unit in \(\H^0(\PP^p,\I^0)\cong \W(F)\) depending on \(\phi\).  To be more precise, choose an isomorphism \(t_h\colon \Oo \rightarrow \omega_h\otimes \Oo(1-p)\).  Define \(t_i:= (\id \times i)^* t_h\) for \(i \in \{0,1\}\), and choose the isomorphisms \(t_x\) and \(t_y\) in the definitions of \(\alpha\) and \(\beta\) above Theorem~\ref{thm:I-ring-2} such that \(t_0 = t_y\) and \(t_1 = \phi^*t_x\). Then  \((\iota_x)_*t_x\) and \((\iota_y)_*t_y\) agree up to multiplication with a unit in \(\W(F)\).  
  
  By combining the identification of \((\iota_x)_*t_x\) with \((\iota_y)_*t_y\) and formula~\eqref{n=2p:cd1}, we find that \(\alpha \beta\) is a generator in degree \(2p\). (Note that  \(s_*\) in \eqref{n=2p:cd1} is an isomorphism by the previous additive computations.) The following identities moreover prove the relation \(\alpha^2 =0\). We write ``\(\sim\)'' for ``equal up to multiplication by a unit'':
  \begin{equation}\label{n=2p;local=0}  
    \begin{array}{llll}
      (\iota_x)_* 1_x \cup (\iota_x)_* 1_x &\sim& (\iota_x)_* 1_x \cup (\iota_y)_* 1_y  & \textnormal{by the \(\mathbb{A}^1\)-homotopy} \\
                                           &=& (\iota_x)_* (1_x \cup \iota_x^* (\iota_y)_* 1_y) & \textnormal{by the projection formula}\\
                                           &=& 0  & \textnormal{since \(\iota_x^* (\iota_y)_* 1_y = 0\) by sequence~\eqref{seq:keysequence-Q_2p}}
    \end{array}  
  \end{equation}
  Similarly, we obtain the relation \(\beta^2 = 0\). 
  
  \medskip
  
  \emph{When \(n=2p\) with even \(p\),}  we consider the following two  \(\AA^1\)-homotopies \(\PP^p \times \AA^1 \rightarrow Q_{2p}\times \AA^1\):
  \begin{align*}
    h([a_0,a_1, \ldots, a_p], t) & := ([ a_0, ta_1,  ta_2, \ldots,  ta_p ; 0,\;  (1-t)a_2,  (t-1)a_1, \ldots, (1-t)a_{p}, (t-1)a_{p-1}],t) \\
    \tilde h([a_0,a_1, \ldots, a_p], t) &:= ([ 0,\; ta_1,  ta_2, \ldots,  ta_p ; a_0 ,  (1-t)a_2,  (t-1)a_1, \ldots, (1-t)a_{p}, (t-1)a_{p-1}],t)
  \end{align*}
  Each homotopy fits into a commutative diagram of the form \eqref{diag:homotopy}.
  Note that \(h_1 = \iota_x\), \(h_0 = \iota_{y'} \phi\), \(\tilde h_1 = \iota_{x'}\) and \(\tilde h_0 = \iota_y\tilde \phi\) for some some linear changes of coordinates \(\phi, \tilde \phi\colon \PP^p \rightarrow \PP^p\).
  
  As in the previous case, we choose isomorphisms \(t\colon \Oo \rightarrow \omega_h\otimes \Oo(1-p)\) and \(\tilde t \colon \Oo \rightarrow \omega_{\tilde h}\otimes \Oo(1-p)\) and we define \(t_i:= (\id \times i)^* t_h\) and \(\tilde t_i := (\id\times i)^* \tilde t\) for \(i \in \{0,1\}\).  Then \((h_1)_*t_1= (h_0)_* t_{0}\) and \((\tilde h_1)_*\tilde t_1 = (\tilde h_0)_* \tilde t_0\) as maps \(\H^{0}(\PP^p, \I^0) \rightarrow \H^{p}(Q_{2p},\I^p,\Oo(p-1))\).   We choose \(t_x\), \(t_y\), \(t_{x'}\) and \(t_{y'}\) such that \(t_1 = t_x\), \(t_0 = \phi^* t_{y'}\), \(\tilde t_1 = t_{x'}\) and \(\tilde t_0 =\tilde \phi^* t_{y}\).  Then we find that  \((\iota_x)_*t_x\) and \((\iota_{y'})_*t_{y'}\) agree up to multiplication by a unit in \(\W(F)\), and likewise \((\iota_{x'})_*\) and \((\iota_y)_*\) agree up to multiplication by a unit.

  From a computation analogous to \eqref{n=2p;local=0} and the above \(\mathbb{A}^1\)-homotopies, we may now deduce the relation \(\alpha \beta = 0\).  Similarly, using the homotopy \(h\) and diagram~\eqref{n=2p:cd1}, we find that \(\alpha^2\) is the additive generator \( s_*(t_0 1 \cup t_1 1)\) in degree~\(2p\).  Using the homotopy \(\tilde h\) and \eqref{n=2p:cd1} with \(x'\) and \(y\) replaced by \(x\) and \(y'\), respectively, we find that \(\beta^2\) is the additive generator \(\tilde{s}_*(\tilde{t}_0 1 \cup \tilde{t}_1 1)\) in degree~\(2p\).  
  (To be precise, in these expressions for \(\alpha^2\) and \(\beta^2\) the isomorphisms ``\(t_i\)'' and ``\(\tilde t_i\)'' really denote the pullbacks of the respective trivializations \(t_i\) from \(\PP^p\) to the base.  For example, for \(t_1 = t_x\) we use the following commutative diagram:
  \[
    \xymatrix{
      \W(F) \ar[d]^-{s_x^*t_x} \ar[r]^-{p^*} & \H^0(\PP^p, \I^0 )\ar[d]^-{t_x} \\
      \W(F, s_x^*\omega_x(p-1)) \ar[r]^-{p^*} & \H^0(\PP^p, \omega_x \otimes \Oo(p-1))
    }
  \]
  However, we drop \(s_x^*\) from our notation for simplicity.)        
  
  To compare \(\alpha^2 = s_*(t_01 \cup t_11)\) and \(\beta^2 = \tilde{s}_*(\tilde{t}_0 1 \cup \tilde{t}_1 1)\), we consider the following involution \(\tau\colon\PP^{2p+1}\to\PP^{2p+1}\):
  \[
    \tau([x_0: x_1: \cdots : x_p ; y_0 : y_1 \cdots : y_p]) := [y_0 : x_1 : \cdots : x_p ; x_0 : y_1 : \cdots : y_p]
  \]
  This involution can be restricted to an involution of \(Q_{2p}\), which we still denote by \(\tau\).  It fits into the following commutative diagram of cartesian squares:
  \[
    \xymatrix{
      &\Spec(F)\ar@{-->}[ddrrrr]^(.3){\tilde{s}} \ar@{=}[dl] \ar[rrrr] \ar[dd] &&&&  \PP^p \ar[dd]_-{h_1} \ar@{=}[dl] \\ 
      \Spec(F) \ar@{-->}[ddrrrr]^>>>>>>>>>{s} \ar[rrrr] \ar[dd]&&&& \PP^p \ar[dd]_(.3){\tilde{h}_1} \\
      &\PP^p \ar@{=}[dl] \ar[rrrr]^{h_0}  &&&& Q_{2p}\ar[dl]^-{\tau}  \\
      \PP^p \ar[rrrr]^-{\tilde{h}_0} &&&& Q_{2p}
    }
  \]
  As an automorphism of \(\PP^{2p+1}\), \(\tau\) induces a canonical isomorphism of line bundles \(\tau^*\Oo(1)\cong \Oo(1)\).  This isomorphism of line bundles restricts to a canonical isomorphism between the respective line bundles \(\Oo(1)\) and \(\tau^*\Oo(1)\) on \(Q_{2p}\), and we implicitly use this canonical identification in the following.
  
  The product \(\tau\times\id\) relates the homotopies \(h\) and \(\tilde h\), in the sense that we also have the following commutative diagram:
  \[
    \xymatrix{ \PP^p \times \AA^1 \ar[r]^-{h} \ar[d]^-{\tilde{h}} &  Q_{2p} \times \AA^1 \ar[dl]^-{\tau\times\id}  \\
      Q_{2p} \times \AA^1}
  \]
  In particular, \(\tau\times\id\) induces a canonical isomorphism of the relative canonical line bundles of \(h\) and \(\tilde h\), \(\Omega_{\tau\times\id}\colon \omega_h \rightarrow \omega_{\tilde{h}}\).  Changing our (previously arbitrary) choices of \(t\) and \(\tilde{t}\) if necessary, we may assume that the following diagram of bundles over \(\PP^p\times\AA^1\) is commutative:
  \[
    \xymatrix{\Oo \ar[r]^-{t} \ar[d]^-{\tilde{t}} &   \omega_h \otimes \Oo(1-p) \ar[dl]^-{\Omega_\tau}  \\
      \omega_{\tilde h} \otimes \Oo(1-p) }
  \]
  Consider the map \(\tau^*\colon \H^{2p}(Q_{2p}, \I^{2p}, \Oo(2p-2)) \rightarrow \H^{2p}(Q_{2p}, \I^{2p}, \Oo(2p-2))\) induced by \(\tau\) and the canonical identification \(\tau^*\Oo(1)\cong\Oo(1)\).  By the base change formula, we find that
  \[
    \tau^* s_* (t_01 \cup t_11) = \tilde{s}_* \id^*(\tilde{t}_0 1 \cup \tilde{t}_1 1)
  \]
  in  \(\H^{2p}(Q_{2p},\I^{2p},\Oo(2p-2))\), i.e.\ \(\tau^*\alpha^2 = \beta^2\).  We claim that the map \(\tau^*\) is precisely \(-\id\), so that we obtain \(\alpha^2 = - \beta^2\).

  To prove the claim, we first note that the pushforward map along the inclusion \(i\colon Q_{2p} \hookrightarrow \PP^{2p+1}\) induces an isomorphism
  \[
    i_*\colon \H^{2p}(Q_{2p}, \I^{2p},\omega_i \otimes i^*\LC) \rightarrow \H^{2p+1}(\PP^{2p+1},\I^{2p+1}, \LC)
  \] 
  for any even twist \(\LC\).  This can be seen from the additive computation above and the identification of the relative canonical line bundle \(\omega_i\) of the inclusion as \(\omega_i \cong \Oo(-2)\). The study of \(\tau^*\) on \(Q_{2p}\) can therefore be translated to the study of \(\tau^*\) on \(\PP^{2p+1}\). Consider a fixed point of \(\tau\) on \(\PP^{2p+1}\):
  \[
    \xymatrix{ \Spec(F) \ar[d]^-{=} \ar[r]^-{\point} & \PP^{2p+1} \ar[d]^-{\tau}\\
      \Spec(F) \ar[r]^-{\point} & \PP^{2p+1 } }
  \]
  The base change formula \ref{thm:base-change-formula} implies that \(\tau^* \point_*\) differs from \(\point_*\) by multiplication with the determinant of the map on normal bundles \(\mathcal{N}_{\point} \rightarrow \mathcal{N}_{\point}\) induced by \(\tau\). It is not hard to see that this determinant is just the determinant of the matrix of the linear change of coordinates \(\tau\), which is \(-1\).
\end{proof}

\section{Geometric bidegrees and real realization}\label{sec:real-realization}

We now compare Theorems \ref{thm:I-ring} and \ref{thm:I-ring-2} with Theorem~\ref{top:integral-ring}. They both apply to the split quadrics \(Q_n\) with \(n=(p,p)\) or \(n=(p+1,p)\) and \(n \geq 3\), while the latter theorem also holds more generally. The \(\ZZ \oplus \ZZ \oplus \ZZII\)-graded ring \(\H^\tot(X, \I)\) contains the \(\ZZ \oplus \ZZII\)-graded subring (even a \(\W(F)\)-algebra) given by the condition that the two \(\ZZ\)-degrees coincide.  This corresponds to motivic bidegrees \((2i,i)\) and hence to Chow and Chow-Witt groups.  Following \cite{HornbostelWendt}, \cite{HWXZ} and others, we denote this important sub-\(\W(F)\)-algebra by \(\H^\bullet(X, \I^\bullet, \Oo \oplus \Oo(1))\).  The computations of Sections \ref{additivecomp} and \ref{multiplicativecomp} easily restrict to these \(\W(F)\)-submodules/subalgebras. For example, restricting Theorem~\ref{thm:I-ring} to \(\H^\bullet(Q_n, \I^\bullet, \Oo \oplus \Oo(1))\) means replacing \(\II\) by \(\W(F)\) and \(\II'\) by \(\FI(F)\).

Now if \(\W(F) \cong \ZZ\), e.g.\ if \(F=\RR\), then there is an explicit ring isomorphism from this graded commutative ring \(\H^\bullet(Q_n, \I^\bullet, \Oo \oplus \Oo(1))\) in algebraic geometry to the graded commutative ring \(\H^\bullet(X(\RR)^{an},\ZZ \oplus \ZZ(1))\) in topology, given by mapping \(\alpha\) to \(\alpha - \beta\) if \(n=2p\), and the identity on all other generators. Moreover, there is another obvious isomorphism between the corresponding rings with \(\ZZII\)-coefficients, and both isomorphisms are compatible via the reduction maps.

None of this is a coincidence, of course. As observed in Remark~\ref{sec:top:alt-additive}, the strategy of our computation of the \(\I\)-cohomology of the algebro-geometric \(Q_n\) also lends itself to a computation of the integral singular cohomology \(\H^\bullet(Q_n,\ZZ \oplus \ZZ(1))\) of the topological real quadric.  Even more is true. We recall from \cite{jacobson} that there is a real realization functor
\[
  \H^\bullet(X,\I^{\bullet}) \to \H^\bullet(X(\RR)^{an},\ZZ)
\]
for a smooth variety \(X\) over \(\RR\), which is known to be an isomorphism in some degrees in special cases, see loc.\ cit.\ and previous work by Fasel and others.  By recent work of the authors and Matthias Wendt \cite{HWXZ}, we have a stronger result for smooth cellular varieties.

\begin{theorem}\label{HWXZ}
  For any smooth variety \(X\) over \(\RR\), the real realization morphism
  \[
    \H^\bullet(X,\I^{\bullet}) \to  \H^\bullet(X(\RR)^{an},\ZZ)
  \]
  is a ring homomorphism.
  If \(X\) is cellular, then this is an isomorphism.
\end{theorem}
\begin{proof}
  See \cite{HWXZ}.
\end{proof}

This applies in particular to the split quadrics \(Q_n\).  The preprint \cite{HWXZ} also studies twisted coefficients and establishes a generalization of Theorem~\ref{HWXZ}, replacing \(\I^{\bullet}\) and \(\ZZ\) by \((\I^{\bullet},\Oo \oplus \Oo(1))\) and \(\ZZ \oplus \ZZ(1)\), respectively.  We deduce that over \(F=\RR\) the description of Theorems~\ref{thm:I-ring} and \ref{thm:I-ring-2} extends to all quadrics covered by Theorem~\ref{top:integral-ring}, in particular to \(\PP^1 \times \PP^1\). However, the Picard group is larger in this case, so we have not computed cohomology with respect to all possible twists for \(\PP^1 \times \PP^1\).  If more generally \(F\) is a subfield of \(\RR\), then the base change composed with real realization is still a ring homomorphism.

\section{Chow-Witt rings}\label{sec:CW-ring}

This section contains a short self-contained presentation and computation of the Chow-Witt ring of \(Q_n\), arising from the fiber product description first used in \cite{HornbostelWendt} that emphasizes the relationship with Chow groups. This description relies on the isomorphism \(\Ch^{\bullet}(X) \cong \H^{\bullet}(X,\Ib^{\bullet})\), and hence on the proof of the Milnor conjecture of Voevodsky et al.

Recall that the \(\ZZ \oplus \ZZ \oplus \ZZII\)-graded ring \(\H^\tot(Q_n, \I)\) restricts to Chow-Witt groups in bidegrees \((i,i)\), corresponding to motivic bidegrees \((2i,i)\).  By Theorem~\ref{CHquadric}, we know that \(\CH^\bullet(Q_{p,p})\) and \(\CH^\bullet(Q_{p,p+1})\) have no two-torsion. Hence to compute \(\CW^\bullet(Q_{p,p})\) and \(\CW^\bullet(Q_{p,p+1})\), we may apply part~(1) of Proposition~2.11 of \cite{HornbostelWendt}, which we briefly recall.

\begin{proposition}\label{prop:HW}
  Let \(X\) be a smooth scheme over a field \(F\) of characteristic \(\neq 2\). The canonical ring homomorphism
  \[
    \CW^\bullet(X)\to \H^\bullet(X,\I^\bullet)\times_{\Ch^\bullet(X)} \ker\partial
  \] 
  induced from the cartesian square defining Milnor-Witt \(K\)-theory is always surjective.  It is injective if \(\CH^\bullet(X)\) has no non-trivial two-torsion. 
\end{proposition}

Here, \(\partial\colon \CH^\bullet(X)\to \H^{\bullet+1}(X, \I^{\bullet+1})\) is defined as in loc.\ cit.
The kernel of \(\partial\) is a subring of \(\CH^\bullet(X)\) because the Bockstein homomorphism satisfies a Leibniz formula \cite[proof of Proposition~4.7]{fasel:chowwittring}.

Note that the general assumption of \cite{HornbostelWendt} that the base field is perfect is not necessary in Proposition~\ref{prop:HW}:  the proof of \cite[Proposition~2.11]{HornbostelWendt} applies verbatim if we use Zariski sheaf cohomology in the ``key diagram'' considered in loc.\ cit., and Zariski sheaf cohomology can be computed from the Gersten complex \cite[Theorem~3.26]{fasel:chowwittring}.
Moreover, Proposition~\ref{prop:HW} also applies to twisted coefficients: the proof of loc.\ cit.\ goes through when twisting with a line bundle \(\LC\). 
  
\begin{lemma}\label{h2-generators-here-vs-DI:Hopf}
  In terms of our chosen generators, the isomorphism of commutative \(\ZZII\)-algebras 
  \(\Ch^\bullet(Q_n) \xrightarrow{\cong} \H^\bullet(Q_n, \Ib^\bullet)\) has the form 
  \begin{align*}
    \bar{x}&\mapsto \bar{\xi}\\
    \bar{y}&\mapsto \begin{cases}
      \overline{\beta} + \delta \overline{\xi}^p & \textnormal{ if \(n\) is even, where \(\delta \in \{0,1\}\)} \\
      \overline{\beta} & \textnormal{ if \(n\) is odd} 
    \end{cases}
  \end{align*}
\end{lemma}
\begin{proof}
  This is obvious, given the degree constraints and that
  we are working over \(\ZZII\).
\end{proof}

Putting everything together, we obtain the following result:

\begin{@empty}
  \renewcommand{\red}{\mathrm{mod}\,2}
  \begin{theorem}\label{thm:CW-ring}
    Let \(Q_n\) be the split quadric over a field \(F\) of characteristic \(\neq 2\) as above, 
    and \(\LC\) a line bundle over \(Q_n\). 
    Then the graded \(\GW(F)\)-module \(\CW^\bullet(Q_{n},\LC)\)
    is given by 
    \[
      \CW^\bullet(Q_n,\LC)\stackrel{\cong}{\to} \H^\bullet(Q_n,\I^\bullet,\LC)\times_{\Ch^\bullet(Q_n)}\ker\partial_{\LC}
    \]
    with \(\ker \partial_{\LC}\) given as follows:

    \begin{align*}
      \text{ for } \LC=\Oo: 
      & 
        \begin{cases}
          \ZZ\langle 2x,2y,xy,x^2,y^2\rangle & \text{ when \(n=2p\), and } p \text{ is even} \\
          \ZZ\langle 2x,x^2,x^p,y\rangle & \text{ when \(n=2p\), and } p \text{ is odd} \\
          \ZZ\langle 2x,x^2,y\rangle & \text{ when \(n=2p+1\), and } p \text{ is even} \\
          \ZZ\langle 2x,2y,xy,x^2,x^p\rangle & \text{ when \(n=2p+1\), and  } p \text{ is odd}
        \end{cases} 
    \end{align*}
    where \(\ZZ\langle \text{elements} \rangle\) denotes the subring of \(\CH^\bullet(Q_n)\) generated by the specified elements (see Theorem~\ref{CHquadric}).
    The following four twisted cases are described as submodules of \(\CH^\bullet(Q_n)\) over the respective four rings above:
    \begin{align*}
      \text{ for } \LC=\Oo(1): 
      & 
        \begin{cases}
          x \cdot \ZZ\langle x^2,xy \rangle  + y \cdot \ZZ + \ker(\red) & \text{ when \(n=2p\), and } p \text{ is even} \\
          x \cdot \ZZ\langle x^2,x^p,y \rangle + \ker(\red) & \text{ when \(n=2p\), and } p \text{ is odd} \\
          x \cdot \ZZ\langle x^2,y,x^{p-1}y\rangle + \ker(\red) & \text{ when \(n=2p+1\), and } p \text{ is even} \\
          x \cdot \ZZ\langle x^2,x^p,xy, x^{p-1}y\rangle + y \cdot \ZZ + \ker(\red) & \text{ when \(n=2p+1\), and } p \text{ is odd} 
        \end{cases}
    \end{align*}
    In this description, \(\red\) denotes the mod-\(2\)-reduction map \(\CH^\bullet(Q_n)\to\Ch^\bullet(Q_n)\).
  \end{theorem}

  \begin{proof}
    Everything except the computation of \(\ker \partial_{\LC}\) has been established already. For the latter, one uses that in the notations of the diagram in \cite[Section~2.4]{HornbostelWendt} we have \(\partial_{\LC}=\beta_{\LC} \circ \red\) and \(\ker \beta _{\LC}=\textnormal{Im} \rho_{\LC}\).  (Strictly speaking, the diagram in \cite[Section~2.4]{HornbostelWendt} does include twists by line bundles, but line bundles can be fitted into the diagram without any difficulty.)  Our computation of the \(\W(F)\)-algebra \(\H^\bullet(Q_{p,q},\I^\bullet)\) and the reduction map \(\rho_{\LC}\) (using Lemma~\ref{h2-generators-here-vs-DI:Hopf}) allow us to completely compute the image of the latter. Note that for \(n=2p\) and \(p\) even this is independent of the value of \(\delta\), using the relation \(\bar{\xi}^{p+1}=0\) both in the untwisted and twisted case and moreover the relation \(y^2 -x^{p}y\) in the twisted case.  From this we easily deduce the kernels of \(\partial_{\LC}\) and \(\beta_{\LC}\) both for \(\LC=\Oo\) and \(\LC=\Oo(1)\).
  \end{proof}
\end{@empty}

For a more concise description of the \(\ZZ\oplus \ZZII\)-graded \(\GW(F)\)-algebra \(\CW^\bullet(Q_{n},\Oo \oplus \Oo(1))\), we introduce an artificial \(\ZZII\)-grading on \(\CH^\bullet(Q_n)\) and \(\Ch^\bullet(Q_n)\) as follows:
\begin{align*}
  \CH^\bullet(Q_n,\Oo\oplus \Oo(1)) := \CH^\bullet(Q_n)[\tau]/(\tau^2-1)\\
  \Ch^\bullet(Q_n,\Oo\oplus \Oo(1)) := \Ch^\bullet(Q_n)[\tau]/(\tau^2-1)
\end{align*}
In this notation, the same arguments as in the proof above yield the following description of the Chow-Witt ring:
\[
  \CW^\bullet(Q_n,\Oo\oplus\Oo(1)) \cong \H^\bullet(Q_n,\I^\bullet,\Oo\oplus\Oo(1))\times_{\Ch^\bullet(Q_n,\Oo\oplus\Oo(1))}\ZZ\langle 2\tau,2x,2y,x\tau,x^p\tau^{p+1},y\tau^{q+1}\rangle,
\]
where \(\ZZ\langle \text{elements} \rangle\) denotes the subring of \(\CH^\bullet(Q_n,\Oo\oplus\Oo(1))\) generated by the specified elements.

\section{Milnor-Witt cohomology}
The additive computations of $\I$-cohomology can  be transferred to Milnor-Witt cohomology.  Again, we begin by recalling the corresponding result for projective spaces:

\begin{theorem}[{\cite[Theorem~11.7]{fasel:ij}}]\label{thm:KMWP}   
  Let \( \PP^n = \PP^n_F\), where \(F\) is a field of characteristic \(\neq 2\). For degrees \(i\in \{0,\dots,n\}\), the Milnor-Witt cohomology groups of \(\PP^n\) are as follows:
  \[	
    \H^i( \PP^n, \KMW_j,\Oo(l) ) \cong
    \begin{cases}
      \K^{MW}_{j-i}(F) & \textnormal{if (\(i = 0\) and \(l \equiv i \mod 2\))}     \\
      & \textnormal{or (\(i = n\) and \(l \not\equiv i \mod 2\))} \\
      \K^{M}_{j-i}(F)  & \textnormal{if \(i \neq 0\) and \(l \equiv i \mod 2\)}    \\
      2\K^{M}_{j-i}(F) & \textnormal{if \(i \neq n\) and \(l \not\equiv i \mod 2\)}
    \end{cases}
  \]
  In degrees \(i\not\in \{0,\dots, n\}\), these groups vanish.
\end{theorem}

Now consider once again our split quadric \(Q_n\) over a field.  Write \(n = p+q\) with \(q = p\) if \(n\) is even and \(q = p+1\) if \(n\) is odd, so that \(q\) is the codimension of \(\iota_x\colon \PPpX \hookrightarrow Q_n\).  Recall from Lemma~\ref{lem:normal-bundles-in-Q} that \(\det \mathcal{N} \cong \Oo(q-1)\) for the normal bundle \(\mathcal{N}\) of this inclusion.  As in $\I$-cohomology, localization and d\'{e}vissage yield a long exact sequence as follows \cite[Corollaire~10.4.10]{faselthesis}:
\[
  \cdots \rightarrow \H^{i-q}(\PP_x^p, \KMW_{j-q}, \LC (1-q) )\xrightarrow{(\iota_x)_*} \H^i(Q_{n}, \KMW_j, \mathcal{L}) \rightarrow \H^i(Q_n - \PP^p_x, \KMW_j, \LC) \rightarrow  \cdots
\]
Homotopy invariance of \(\I\)-cohomology (Theorem~\ref{thm:homotopy-invariance}) 
together with the corresponding result for Milnor cohomology and the usual
five lemma argument implies homotopy invariance for Milnor-Witt cohomology.
Applying this to the affine space bundle \(\rho\colon Q_n - \PP^p_x \rightarrow \PP^p_y\), we obtain the following exact sequence:
\begin{equation}\label{eq:KMWQ}
  \cdots \rightarrow \H^{i-q}(\PP_x^p, \KMW_{j-q}, \LC(1-q) )\xrightarrow{(\iota_x)_*} \H^i(Q_{n}, \KMW_j, \mathcal{L}) \xrightarrow{\iota^*_y} \H^i(\PP^p_y, \KMW_j, \LC) \rightarrow  \cdots
\end{equation}  

\begin{theorem}
  Let \(Q_n\) be the \(n\)-dimensional split quadric over a field \(F\) of characteristic \(\neq 2\), with \(n \geq 3\).  Let \(\LC\) be a line bundle over \(Q_n\).  The exact sequence (\ref{eq:KMWQ}) above splits, so that we have isomorphisms
  \[
    \H^i(Q_{n}, \KMW_j, \mathcal{L})  
    \cong 
    \H^{i-q}(\PP_x^p, \KMW_{j-q}, \LC(1-q) ) \oplus  \H^i(\PP^p_y, \KMW_j, \LC) 
  \]
\end{theorem}

\begin{proof}
  Suppose \(n = 2p+1\).  Then for degree reasons the map \(\iota_y^*\) in sequence~\eqref{eq:KMWQ} is an isomorphism in the range \(0\leq i < p\), and the map \((\iota_x)_*\) is an isomorphism in the range \(p+1 <  i \leq n\).  In degrees \(i = p\) and \(i = p+1\), we have an exact sequence
  \begin{equation}\label{eq:KMWQ-n-odd}
    0 \rightarrow \H^p(Q_n, \KMW_{j}, \LC) \stackrel{\iota^*_y}\rightarrow \H^p(\PP^p_y, \KMW_{j}, \LC ) \stackrel{\partial}\rightarrow \H^0(\PP_{x}^p, \KMW_{j-q}, \LC \otimes  \omega_x) \stackrel{(\iota_x)_*}\longrightarrow \H^{p+1}(Q_n, \KMW_{j}, \LC ) \rightarrow 0
  \end{equation}
  Consider first the case \(\LC = \Oo(p+1)\).  Consider the following commutative diagram induced by the short exact sequences of sheaves \(0\to \I^{j+1} \to \KMW_j \to \KM_j\to 0\) (cf.\  \cite[above Lemma~11.3]{fasel:ij}),
  keeping in mind that the twists with respect to any \(\LC\) are
  invisible to \(\KM_j\):  
  \[ 
    \xymatrix{
      \H^{p}(Q_n,\I^{j+1}, \LC) \ar[r] \ar[d]^-{\iota_y^*}_-\cong &\H^{p}(Q_n,\KMW_j, \LC) \ar[r] \ar[d]^-{\iota_y^*}
      & \H^{p}(Q_n,\KM_j)    \ar[r] \ar@{->>}[d]^-{\iota_y^*} & \H^{p+1}(Q_n,\I^{j+1},\LC) \ar[d] \\
      \H^{p}(\PP^p_y,\I^{j+1}, \LC) \ar[r]  &\H^{p}(\PP^p_y,\KMW_j,\LC) \ar[r] 
      & \H^{p}(\PP^p_y,\KM_j)    \ar[r] & 0
    }
  \]
  From Theorem~\ref{thm:I-additive}, we know that the group \(\H^{p+1}(Q_n,\I^{j+1}, \Oo(p+1))\) in the top right corner vanishes.  Our proof of Theorem~\ref{thm:I-additive} moreover shows that \(\iota_y\) induces an isomorphism on \(\I\)-cohomology in degree \(p\), so the vertical map on the far left is an isomorphism.  Recall from \cite[Theorem~11.1]{fasel:ij} that \(\H^p(\PP^p_y, \KM_j) = \K^M_{j-p}(F) \cdot e(\Oo(1))^p\).  By considering the Euler class of the bundle \(\Oo(1)\) over \(Q_n\) (and using Lemma~\ref{lem:picquadrics}), we find that \(\iota_y^*\) is surjective on \(\KM\)-cohomology in degree~\(p\).  Altogether, it follows from a version of the five lemma that the second vertical map in the diagram is also surjective. So in sequence \eqref{eq:KMWQ-n-odd}, \(\partial = 0\)  and both \(\iota_y^*\) and \((\iota_x)_*\) are isomorphisms.
  
  Now consider the case \(\LC = \Oo(p)\). Recall from Lemma~\ref{lem:normal-bundles-in-Q} that \(\det (\mathcal{N})^\vee = \Oo(-p)\). Consider the following commutative diagram induced by the same short exact sequence of sheaves as in the previous case:
  \[ 
    \xymatrix{ 	
      0 \ar[r] \ar@{>->}[d]  & \H^{0}(\PP^p_x,\I^{j-q+1}) \ar[r]  \ar[d]^-{(\iota_x)_*}_-\cong & \H^{0}(\PP^p_x,\KMW_{j-q}) \ar[r]  \ar[d]^-{(\iota_x)_*}
      & \H^{0}(\PP^p_x,\KM_{j-q})     \ar[d]^-{(\iota_x)_*}_-\cong \\
      \H^p(Q_n,\KM_j) \ar[r]^-\delta & \H^{p+1}(Q_n,\I^{j+1}, \Oo(p)) \ar[r]                                    & \H^{p+1}(Q_n,\KMW_j, \Oo(p)) \ar[r] 
      & \H^{p+1}(Q_n,\KM_j)      	
    }
  \]
  We have isomorphisms on the left and right in this diagram by the computation of \(\I\)-cohomology and Milnor cohomology of quadrics above. (One sees this directly from the localization sequence and degree considerations.)  We claim that the boundary map \(\delta\) vanishes, and that therefore \((\iota_x)_*\) is injective on \(\KMW\) cohomology. To see this, we compare \(\delta\) with the boundary map in the long exact sequences associated with the coefficient sequence \(0\to \I^{j+1} \to \I^j \to \Ib^j \to 0\), via the following commutative diagram (see \cite[Lemma~11.3]{fasel:ij}):
  \begin{equation*}
    \xymatrix{ 
      {\quad\quad\dots\quad\quad} \ar[r] & \H^p(Q_n,\KM_j) \ar[r]^-\delta \ar[d] & \ar@{=}[d]	\H^{p+1}(Q_n,\I^{j+1}, \Oo(p)) \ar[r] & {\quad\quad\dots\quad\quad} \\
      \H^p(Q_n,\I^j,\Oo(p)) \ar[r] & \H^p(Q_n,\Ib^j) \ar[r]^-{\partial}                               & \H^{p+1}(Q_n,\I^{j+1}, \Oo(p)) \ar[r] & {\quad\quad\dots\quad\quad}
    }
  \end{equation*}
  Our computations of $\I$- and \(\Ib\)-cohomology show that the reduction map on the left of the lower sequence is surjective for odd-dimensional quadrics. So the boundary map \(\partial\) is zero, and hence so is \(\delta\).
  
  Suppose now that \(n = 2p\). In this case, the map \(\iota_y^*\) in sequence~\eqref{eq:KMWQ} is an isomorphism for degree reasons in the range \(0\leq i < p-1\), and the map \((\iota_x)_*\) is an isomorphism for degree reasons in the range \(p+1 <  i \leq n\). It remains to show that \(\iota_y^*\) is also an isomorphism in degree \(p-1\), that \((\iota_x)_*\) is an isomorphism in degree \(p+1\), and that the following sequence in degree \(p\) is split exact:
  \[ 
    \xymatrix{
      0 \ar[r] & \H^0(\PPpX,\KMW_{j-p}, \omega_x \otimes \LC) \ar[r]^-{(\iota_x)_*} &\H^p(Q_n,\KMW_j, \LC) \ar[r]^-{\iota_y^*} 
      & \H^p(\PPpY,\KMW_j, \LC)    \ar[r] & 0 
    }
  \]
  To verify that this sequence is indeed split exact, one can use the cartesian diagram~\eqref{n2p:cartesian} and the base change formula.  The other two claims, concerning degrees \(p-1\) and \(p+1\), then follow. 
\end{proof}

\begin{corollary}
  \newcommand*{\tn}[1]{\textnormal{#1}}
  The Milnor-Witt cohomology groups of a split quadric \(Q_n\) over a field \(F\) of characteristic \(\neq 2\) in degrees \(i\in\{0,\dots,n\}\) are as follows:
  
  Case \(n = 2p\):
  \begin{align*}
    \H^i( Q_{2p}, \KMW_j, \Oo(l) ) 
    & \cong
      \left\{
      \begin{alignedat}{12}
        & \K^{MW}_{j-i}(F)                                      & \tn{ if } & \tn{(} & i & = 0            & \quad & \tn{ and } & \quad l & \equiv i     &  & \mod 2 \tn{)} \\
        &                                                               & \tn{ or } & \tn{(} & i & = n            &       & \tn{ and } & l       & \equiv i     &  & \mod 2 \tn{)} \\
        & \K^{MW}_{j-i}(F)\oplus \K^{MW}_{j-i}(F) \quad & \tn{ if } &        & i & = p            &       & \tn{ and } & l       & \not\equiv i &  & \mod 2        \\
        & \K^{M}_{j-i}(F)\oplus 2\K^{M}_{j-i}(F) \quad  & \tn{ if } &        & i & = p            &       & \tn{ and } & l       & \equiv i     &  & \mod 2        \\
        & \K^{M}_{j-i}(F)                                       & \tn{ if } & \tn{(} & 0 & < i < p        &       & \tn{ and } & l       & \equiv i     &  & \mod 2 \tn{)} \\
        &                                                               & \tn{ or } & \tn{(} & p & < i     \leq n &       & \tn{ and } & l       & \not\equiv i &  & \mod 2 \tn{)} \\
        & 2\K^{M}_{j-i}(F)                                      & \tn{ if } & \tn{(} & 0 & \leq i  < p    &       & \tn{ and } & l       & \not\equiv i &  & \mod 2 \tn{)} \\
        &                                                               & \tn{ or } & \tn{(} & p & < i     < n    &       & \tn{ and } & l       & \equiv i     &  & \mod 2 \tn{)} \\
      \end{alignedat}
    \right.
    \intertext{
    Case \(n = 2p + 1\):
    }
    \H^i( Q_{2p+1}, \KMW_j, \Oo(l) ) 
    & \cong
      \left\{
      \begin{alignedat}{12}
        & \K^{MW}_{j-i}(F) \quad\quad\quad\quad\quad & \tn{ if } & \tn{(} & i          & = 0    & \quad & \tn{ and } & \quad l & \equiv i     &  & \mod 2 \tn{)} \\
        &                              & \tn{ or } & \tn{(} & i          & = p    &       & \tn{ and } & l       & \not\equiv i &  & \mod 2 \tn{)} \\
        &                              & \tn{ or } & \tn{(} & i          & = p+1  &       & \tn{ and } & l       & \not\equiv i &  & \mod 2 \tn{)} \\
        &                              & \tn{ or } & \tn{(} & i          & = n    &       & \tn{ and } & l       & \equiv i     &  & \mod 2 \tn{)} \\
        &\K^{M}_{j-i}(F)        & \tn{ if } & \tn{(} & 0 &< i       \leq p &       & \tn{ and } & l       & \equiv i     &  & \mod 2 \tn{)} \\
        &                               & \tn{ or } & \tn{(} & p+1 &< i     \leq n &       & \tn{ and } & l       & \not\equiv i &  & \mod 2 \tn{)} \\
        &2\K^{M}_{j-i}(F)       & \tn{ if } & \tn{(} & 0 &\leq i    < p    &       & \tn{ and } & l       & \not\equiv i &  & \mod 2 \tn{)} \\
        &                               & \tn{ or } & \tn{(} & p+1 &\leq i  < n    &       & \tn{ and } & l       & \equiv i     &  & \mod 2 \tn{)} \\
      \end{alignedat}
    \right.
  \end{align*}
  In degrees \(i\not\in \{0,\dots,n\}\), all groups are zero.
\end{corollary}

\end{document}